\documentclass[11pt]{article}
\usepackage{amsmath,amssymb,amsthm}
\usepackage{latexsym}
\usepackage[dvipdfmx]{graphicx}

\usepackage{bm}

\usepackage{color}

\def\cd{\stackrel{d}{\longrightarrow}} % convergence in law
\def\Rbb{\mathbb{R}}

\def\<{\langle}
\def\>{\rangle}

\def\st{\text{\rm s.\,t.\ }}

\newtheorem{theorem}{Theorem}

\newtheorem{lemma}{Lemma}

\newtheorem{remark}{Remark}

\def\0{\mbox{\bf 0}}
\def\1{\mbox{\bf 1}}
\def\2{\mbox{\bf 2}}
\def\3{\mbox{\bf 3}}
\def\4{\mbox{\bf 4}}
\def\5{\mbox{\bf 5}}
\def\6{\mbox{\bf 6}}
\def\7{\mbox{\bf 7}}
\def\8{\mbox{\bf 8}}
\def\9{\mbox{\bf 9}}

\def\b{{\bm b}}

\def\s{{\bm s}}

\def\v{{\bm v}}

\def\x{{\bm x}}

\def\z{{\bm z}}

\title{Robust Estimation under Heavy Contamination using\\ Enlarged Models}
\date{}
\author{T. Kanamori and H. Fujisawa}

\begin{document}
\maketitle

\noindent

\begin{abstract}
In data analysis, contamination caused by outliers is inevitable, and
robust statistical methods are strongly demanded. 
In this paper, our concern is to develop a new approach for robust data analysis 
based on scoring rules. 
The scoring rule is a discrepancy measure to assess the quality of probabilistic forecasts. 
We propose a simple way of estimating not only the parameter in the statistical model but
also the contamination ratio of outliers. 
Estimating the contamination ratio is important, since 
one can detect outliers out of the training samples based on the estimated contamination ratio. 
For this purpose, we use scoring rules with an extended statistical models, 
that is called the enlarged models. 
Also, the regression problems are considered. 
We study a complex heterogeneous contamination, in which 
the contamination ratio of outliers in the dependent variable may depend on the independent variable. 
We propose a simple method to obtain a robust regression estimator under 
heterogeneous contamination. 
In addition, we show that our method provides also an estimator of the expected
contamination ratio that is available to detect the outliers out of training samples. 
Numerical experiments demonstrate the effectiveness of our methods compared to the
conventional estimators. 
\end{abstract}

\section{Introduction}
\label{sec:Introduction}
In the big data era, robust data analysis is becoming more important than before. 
Nowadays, collecting a large dataset such as the data on the web 
is an easily task, while the quality of data may not be properly controlled. 
In such dataset, contamination caused by outliers such as 
incorrectly measured, or mis-recorded samples will be inevitable. 
Hence, robust statistical methods are demanded to extract valuable information from  
ubiquitous contaminated data. 
The concept of outliers is elusive, and it will be difficult to establish a reliable
statistical model for outliers. 
Hence, robust statistical methods are expected to automatically reduce the effect of
outliers. 

Robust statistics has a long history, and a lot of promising estimators 
were proposed. 
It is well known that the maximum likelihood estimator (MLE) suffers from 
a detrimental effect from outliers. 
The MLE can have a large bias even under a single erroneous observation. 
Many robust estimators were developed to reduce the bias induced by outliers. 
The statistical properties of robust estimators were deeply investigated 
by developing useful concepts such as the influence function, gross-error sensitivity,
break-down point, and so forth; 
see \cite{Hampel_etal86,huber64:_robus,maronna06:_robus_statis} for details. 

A way to reduce the effect of outliers is to employ weighted estimators, 
in which weight is introduced on each training sample. 
In the estimation procedure, the weight on the outlier is automatically reduced to make
the estimator stable.  
The weighted estimators are regarded as an extension of the MLE that has a constant
weight. 
Basu et al.~\cite{a.98:_robus_effic_estim_minim_densit_power_diver,
basu10:_statis_infer,jones01:_compar}
proposed the robust estimator based on the density-power weight, 
and Fujisawa and Eguchi~\cite{fujisawa08:_robus} 
introduced another type of weighting scheme to 
deal with heavily contaminated data. 

The weighted estimators are closely related to the scoring rules % and divergence measures
defined on the set of probability densities. 
The scoring rule is a quantity to assess the quality of 
probabilistic forecasts~\cite{gneiting07:_stric_proper_scorin_rules_predic_estim}. 
The MLE corresponds to the Kullback-Leibler score, 
and the weighted estimators introduced in the above are derived from 
the density-power score or gamma-score. 
From the standpoint of scoring rules, 
a unified framework of weighted estimators is recently presented 
by Kanamori and Fujisawa~\cite{kanamoriar:_affin_invar_diver_compos_scores_applicy}, 
in which a new class of scoring rule called H\"{o}lder score was proposed. 

In this paper, our concern is to develop a new approach for robust data analysis 
based on scoring rules. 
Usually, the scoring rule is defined as a functional on the set of probability densities. 
However, there are a lot of scoring rules that can be defined over a set of non-negative
functions. 
Exploiting such scoring rules, we propose a simple way of estimating not only the
parameter in the statistical model but also the contamination ratio of outliers. 
Estimating the contamination ratio is important to detect outliers out of the training 
samples. Indeed, one can identify the outliers by picking up the estimated number of training
samples in ascending order of the estimated value of the target probability density. 
For this purpose, we use scoring rules with an enlarged extension of statistical
models, that is called the enlarged models. 

We apply the proposed method to regression problems. 
For each independent variable $x$, the dependent variable $y$ 
may be contaminated. 
When the contamination ratio on $y$ does not depend on $x$, 
i.e., the situation of homogeneous contamination, 
the problem is almost the same as the robust estimation of the probability density. 
On the other hand, when the contamination ratio depends on $x$, i.e., 
the heterogeneous contamination, the situation is rather complex. 
We propose a simple method to obtain a robust regression estimator under 
heterogeneous contamination. In addition, our method provides the estimator of the
contamination ratio to detect the outliers out of training samples. 
In our approach, a scoring rule is used with an enlarged location-scale models. 
We prove that our methods has a small bias even under complex 
heterogeneous contamination. 
Moreover, we show that our estimator efficiently works even when 
both independent and dependent variables are heavily contaminated. 

The remainder of the article is organized as follows. 
In Section~\ref{sec:ScoringRules}, we introduce some scoring rules for the statistical
inference. 
In Section~\ref{eqn:Scale_Models_Heavy_Contaminations}, 
we propose statistical methods using enlarged models. 
We demonstrate how our estimator works to estimate not only the model parameter but also
the contamination ratio. 
In Section~\ref{sec:Regression_Problems}, 
the proposed method is applied to regression problems. 
We show that our approach efficiently works even under heterogeneous contamination. 
To confirm the practical efficiency of our methods, we present some numerical experiments in 
Section~\ref{sec:Numerical_Experiments}. 
In Section~\ref{sec:Conclusion}, we close this article with a discussion of the 
possibility of the newly introduced estimation methods. 
Technical calculations and proofs are found in the appendix. 

Let us summarize the notations to be used throughout the paper: 
Let $\Rbb$ be the set of all real numbers, and $\Rbb^n$ denotes the $n$-dimensional
Euclidean space. The univariate normal distribution with the mean $\mu$ and variance
$\sigma^2$ is denoted as $N(\mu,\sigma^2)$, and the $d$-dimensional multivariate normal
distribution with the mean vector $\mu$ and variance-covariance matrix $\Sigma$ is
expressed as $N_d(\mu,\Sigma)$. 
For the function $f(x)$, the integral $\int{}f(x)dx$ is often denoted as $\<f\>$.

\section{Scoring Rules}
\label{sec:ScoringRules}
Scoring rule is a class of discrepancy measures between two probability distributions, and 
it is widely used to statistical 
inference~\cite{brier50:_verif,%
gneiting07:_stric_proper_scorin_rules_predic_estim,%
good71:_commen_measur_infor_uncer_r,%
murata04:_infor_geomet_u_boost_bregm_diver,%
parry12:_proper_local_scorin_rules}. 
In this section, we briefly introduce some scoring rules: the density-power score,
pseudo-spherical score, and H\"{o}lder score, and show some statistical properties. 
The density-power score and pseudo-spherical score are used for robust parameter
estimation. 
The H\"{o}lder score is a class of extended scoring rules including these scoring rules.

\subsection{Density-power score and pseudo-spherical score}
\label{subsec:density-power-pseudo-spherical}
First of all, we briefly review the scoring rules. 
See \cite{gneiting07:_stric_proper_scorin_rules_predic_estim} for details. 
Let $p(x)$ and $q(x)$ be probability densities on the Euclidean space $\Rbb^k$, and
$\ell(x,q)$ be a real-valued function of the point $x\in\Rbb^k$ 
and probability density $q$. 
For the probability densities $p$ and $q$, 
the scoring rule is a real-valued function $S(p,q)$ expressed as 
\begin{align*}
 S(p,q)=\int{\!}p(x)\,\ell(x,q)dx. 
\end{align*}
The scoring rule is said to be  \emph{proper}, 
if the inequality $S(p,q)\geq{}S(p,p)$ holds for arbitrary probability densities 
$p$ and $q$ as long as the integral exists.  Moreover, 
if the equality $S(p,q)=S(p,p)$ leads to $p=q$ almost surely, 
$S(p,q)$ is called the \emph{strictly} proper scoring rule. 
The strictly proper scoring rule $S(p,q)$ defines the divergence $D(p,q)=S(p,q)-S(p,p)$,
that is an extension of squared distance measures on the space of probability densities. 
One of the most popular strictly proper scoring rules is 
the Kullback-Leibler (KL) score, that is defined from $\ell(x,q)=-\log{q(x)}$. 
The divergence associated with the KL score is nothing but the KL divergence. 

We can use strictly proper scoring rules for statistical inference. 
Let $p_\theta(x)$ be a parametrized probability density by the parameter $\theta$, where 
$\theta$ is a member of an open subset $\Theta$ in $\Rbb^d$. 
When the i.i.d. samples $x_1,\ldots,x_n$ are observed from the probability density $p$, 
the statistical model $p_\theta$ is used to estimate the density $p$ based on the samples. 
We assume that $p$ is realized by a probability density in the model. 
Then, the minimization of the empirical loss, 
\begin{align*}
\min_{\theta\in\Theta}\,\frac{1}{n}\sum_{i=1}^{n}\ell(x_i,p_\theta), 
\end{align*}
is expected to provide a good estimate of the probability density $p$. 
This is because the empirical mean converges in probability to the score $S(p,p_\theta)$,
that is minimized at $p_\theta=p$. 

Let us introduce two strictly proper scoring rules; one is the density-power score
$S_{\mathrm{power}}(p,q)$ and the other is the pseudo-spherical score
$S_{\mathrm{sphere}}(p,q)$. 
Both scores have a positive real parameter $\gamma$. 
Given $\gamma>0$, the density-power score is defined as 
\begin{align*}
 S_{\mathrm{power}}(p,q)
 &=
 \gamma
 \<q^{1+\gamma}\>-(1+\gamma)\<pq^\gamma\>, 
\end{align*}
and the associated loss function is given as 
\begin{align*}
 \ell(x,q)=
 \gamma\<q^{1+\gamma}\>-(1+\gamma)q(x)^\gamma. 
\end{align*}
See~\cite{a.98:_robus_effic_estim_minim_densit_power_diver,basu10:_statis_infer} 
for details of the density-power score and its applications. 
On the other hand, 
the pseudo-spherical score~\cite{good71:_commen_measur_infor_uncer_r} is defined as 
\begin{align*}
 S_{\mathrm{sphere}}(p,q)
 &=
 -\frac{\<pq^\gamma\>}{\<q^{1+\gamma}\>^{\gamma/(1+\gamma)}}, 
\end{align*}
that is derived from the loss function, % $\ell(x,q)$, 
\begin{align*}
 \ell(x,q)=-\frac{q(x)^\gamma}{\<q^{1+\gamma}\>^{\gamma/(1+\gamma)}}. 
\end{align*}
The H\"{o}lder's inequality assures that 
the pseudo-spherical score is the strictly proper scoring rule 
%on for the probability densities. 
The monotone transformation $-\frac{1}{\gamma}\log(-S_{\mathrm{sphere}}(p,q))$ 
is called gamma cross entropy. 
The statistical property of the estimator based on gamma cross entropy was investigated 
in~\cite{fujisawa08:_robus}. 
As the parameter $\gamma$ tends to zero, the estimator derived from the density-power
score or pseudo-spherical score gets close to the MLE; 
see \cite{a.98:_robus_effic_estim_minim_densit_power_diver,fujisawa08:_robus} 
for details. 

For some scoring rules, their domain can be extended to a set of non-negative functions. 
Indeed, one can confirm that 
for non-negative and non-zero functions $f$ and $g$, 
the inequalities 
$S_{\mathrm{power}}(f,g)\geq{}S_{\mathrm{power}}(f,f)$ and 
$S_{\mathrm{sphere}}(f,g)\geq{}S_{\mathrm{sphere}}(f,f)$ hold. 
For the density-power score, 
the equality $S_{\mathrm{power}}(f,g)=S_{\mathrm{power}}(f,f)$ 
for non-negative functions leads to $f=g$. 
For the pseudo-spherical score, however, 
the equality 
$S_{\mathrm{sphere}}(f,g)=S_{\mathrm{sphere}}(f,f)$ 
holds if $f$ and $g$ are linearly dependent. 
Note that the linearly dependent probability densities should be identical. 
Hence, the pseudo-spherical score is strictly proper on the set of probability densities, 
while it is not strictly proper on the set of non-negative functions. 
For the reader's convenience, we give a self-contained short proof
of the above facts in Appendix~\ref{appendix:preliminaries}. 

\subsection{Robustness of Estimators based on Scoring Rules}
\label{subsec:Estimation_under_HeavyContainations}
Let us introduce the robustness property of estimators based on the above scoring rules. 
Suppose that our target is to estimate the probability density $p_0(x)$ 
from the observed samples. 
We use the parametric statistical model $p_\theta(x), \theta\in\Theta$ for the estimation
of the target density $p_0(x)$. 
Assume that $p_0(x)=p_{\theta_0}(x)$ holds for $\theta_0\in\Theta$, 
i.e., the target is realized by the model. 
Let $w(x)$ be a probability density of contamination. 
Suppose that the observations $x_1,\ldots,x_n$ are drawn from 
the contaminated probability density, 
\begin{align}
 \label{eqn:heavy-contamination}
 p(x)=c_0p_0(x)+(1-c_0)w(x), 
\end{align}
in which $1-c_0$ is the contamination ratio that typically lies in the interval $[0,1/2)$. 
Here, we do not assume that $1-c_0$ is infinitesimal, i.e., 
we deal with the situation of heavy contamination. 
Instead, we assume that for a positive constant $\gamma$, 
the quantity
\begin{align*}
 \varepsilon_\theta=\<wp_\theta{}^\gamma\>=\int{}w(x)p_\theta(x)^\gamma{dx}
\end{align*}
is sufficiently small around $\theta=\theta_0$. 
This assumption indicates that the contamination density $w(x)$ mostly lies on the tail of
the target density $p_0(x)$. 

Let us consider the estimation of the target density under heavy contamination. 
The empirical probability density is denoted as $\widetilde{p}(x)$, 
that is expressed by the sum of Dirac's delta function. 
The empirical pseudo-spherical score on the model,
$S_{\mathrm{sphere}}(\widetilde{p},p_\theta)$, converges in probability to
$S_{\mathrm{sphere}}(p,p_\theta)$, and we have 
\begin{align*}
 S_{\mathrm{sphere}}(p,p_\theta)
 &=
 c_0S_{\mathrm{sphere}}(p_0,p_\theta)+
 (1-c_0)\<p_\theta^{1+\gamma}\>^{-\gamma/(1+\gamma)}\varepsilon_\theta. 
\end{align*}
Since $\varepsilon_\theta$ is assumed to be sufficiently small around $\theta=\theta_0$, 
the optimal solution of $\min_\theta{}S_{\mathrm{sphere}}(p,p_\theta)$ will be 
close to that of $\min_\theta{}S_{\mathrm{sphere}}(p_0,p_\theta)$. 
Hence, even under heavy contamination, the pseudo-spherical score produces
approximately consistent estimator of the target density $p_0$. 
The argument above was presented in~\cite{fujisawa08:_robus}. 

For the density-power score, the same argument does not hold. 
Indeed, we have 
\begin{align}
 \label{eqn:contaminated-power-score}
 S_{\mathrm{power}}(p,p_\theta) &=
 S_{\mathrm{power}}(c_0p_0,p_\theta)-(1+\gamma)(1-c_0)\varepsilon_\theta. 
\end{align}
Even if $\varepsilon_\theta$ is exactly zero, the minimizer of 
$\min_\theta{}S_{\mathrm{power}}(c_0p_0,p_\theta)$ will not be 
equal to $\theta=\theta_0$. 
Hence, the density-power score does not produce 
the approximately consistent estimator under heavy contamination.

\subsection{H\"{o}lder score}
\label{subsec:Holder}
As an extension of the scoring rules, Kanamori and Fujisawa proposed 
the H\"{o}lder score that is derived from the invariance 
under data transformations~\cite{kanamoriar:_affin_invar_diver_compos_scores_applicy}. 
The H\"{o}lder score includes the density-power score and pseudo-spherical score 
as special cases. 

Let us define the H\"{o}lder score. 
For a real-valued function $\phi(z)$ defined for $z\geq0$, 
suppose that $\phi(1)=-1$ and $\phi(z)\geq-z^{1+\gamma}$, 
where $\gamma$ is a positive real constant. 
Given $\gamma>0$, the H\"{o}lder score $S_\phi$ based on the function $\phi$ 
is defined as 
\begin{align}
 \label{eqn:Holder-score}
 S_\phi(f,g)=\phi
 \left( \frac{\<fg^\gamma\>}{\<g^{1+\gamma}\>} \right){\<g^{1+\gamma}\>}
\end{align}
for the non-negative functions $f$ and $g$. 
The H\"{o}lder inequality assures that $S_\phi(f,g)\geq{}S_\phi(f,f)$ holds, and 
the equality $S_\phi(f,g)=S_\phi(f,f)$ leads to the linear dependence of $f$ and $g$. 
More involved argument yields that 
for probability densities $p$ and $q$, the equality $S_\phi(p,q)=S_\phi(p,p)$ leads to
$p=q$; see~\cite{kanamoriar:_affin_invar_diver_compos_scores_applicy} for details. 
We give a self-contained short proof of the above facts 
in Appendix~\ref{appendix:preliminaries}. 
Generally, the H\"{o}lder score $S_\phi(p,q)$ for the probability densities $p$ and $q$  
is not expressed as the expectation with respect to $p$. 
However, one can substitute the empirical distribution of training samples into~$p$, 
since $S_\phi(p,q)$ depends on $p$ through the integral $\<pq^\gamma\>$. 
The H\"{o}lder score with $\phi(z)=\gamma-(1+\gamma)z$ is reduced to the density-power
score, and the lower bound $\phi(z)=-z^{1+\gamma}$ yields that
$S_\phi(f,g)=-(-S_{\mathrm{sphere}}(f,g))^{1+\gamma}$. 

The H\"{o}lder score is derived from the invariance property of the data transformation. 
Suppose that the probability density $p(x)$ is transformed to $\bar{p}(z)$, 
when the data $x$ is changed to $z$ by an affine transformation. 
Then, the divergence $S_\phi(\bar{p},\bar{q})-S_\phi(\bar{p},\bar{p})$ is 
converted into ${h}\{S_\phi(p,q)-S_\phi(p,p)\}$, 
where $h$ is a positive constant depending only on the affine transformation of data. 
This implies that the data transformation does not essentially change the distance
structure on the set of probability densities. 
In addition, the affine invariance implies that the estimator defined from the H\"{o}lder
score is equivariant~\cite{berger85:_statis_decis_theor_bayes_analy}. 
In other words, the estimator does not essentially depend on the choice of the system of
units in the measurement. 
This is a desirable property for statistical data analysis.

\section{Robust Estimation using Enlarged Models}
\label{eqn:Scale_Models_Heavy_Contaminations}
Detecting outliers out of training samples is an important task in data analysis. 
To deal with this issue, we introduce estimators of the contamination ratio based on 
scoring rules with enlarged models. 
We present some theoretical properties of the proposed estimators. 

\subsection{Contamination Ratio Estimation using Enlarged Models}
\label{subsec:Scoring_Rules_Scaled_Models}
As shown in the previous section, the estimator based on the pseudo-spherical score 
produces an approximately consistent estimator of the target density even 
under heavy contamination. 
However, the ratio $c_0$ in the contaminated distribution is not estimated. 
Estimating the contamination ratio $1-c_0$ is available to detect outliers 
out of the training samples. 
Using the estimated contamination ratio, 
one can identify the outliers out of the training samples by picking up the estimated 
number of training samples in ascending order of the estimated value of the target 
probability density. 
This is because the outliers are assumed to mostly lie on the tail of the underlying
target density. 

%The parameter $c_0$ that determines the contamination ratio is not estimated. 
In order to estimate not only the target density but also the contamination ratio, 
we use the \emph{enlarged model} $m_{\xi}(x)$ defined as 
\begin{align*}
 m_{\xi}(x)=cp_\theta(x), \quad \xi=(c,\theta),\ c>0,\ \theta\in\Theta, 
\end{align*}
where $p_\theta(x)$ is a parametrized probability density and $c$ is 
a one-dimensional positive real parameter to estimate the ratio $c_0$. 

Let us consider the estimator based on the density-power score 
with the enlarged model. 
Suppose that the samples are drawn from the contaminated probability density
\eqref{eqn:heavy-contamination}, and that $p_0=p_{\theta_0}$ holds. 
In the same way as \eqref{eqn:contaminated-power-score}, we have
\begin{align}
 \label{eqn:diff-power-contamination}
 S_{\mathrm{power}}(p,cp_\theta)
 =
 S_{\mathrm{power}}(c_0p_0,cp_\theta)-(1+\gamma)(1-c_0)c^\gamma\varepsilon_\theta
\end{align}
for the enlarged model $cp_\theta$. 
If $\varepsilon_\theta$ is sufficiently small around $\theta=\theta_0$, 
the optimal solution of the problem $\min_{c,\theta} S_{\mathrm{power}}(p,cp_\theta)$
will be close to that of the problem 
$\min_{c,\theta}S_{\mathrm{power}}(c_0p_0,cp_\theta)$. 
Remember that the density-power score is strictly proper on the set of non-negative
functions. 
Therefore, the density-power score with the enlarged model enables us to 
estimate both the target density $\theta_0$ and the ratio $c_0$. 
On the other hand, the argument in the above is not valid for the pseudo-spherical score, 
because $S_{\mathrm{sphere}}(p,cp_\theta)=S_{\mathrm{sphere}}(p,p_{\theta})$ 
holds for all $c>0$.

H\"{o}lder scores with some regularity conditions
are also available to estimate the target density and contamination ratio. 
Indeed, when $\varepsilon_\theta$ is sufficiently small around $\theta=\theta_0$, 
the H\"{o}lder score defined from a smooth function $\phi$ satisfies 
\begin{align*}
 S_\phi(p,cp_\theta)
& =
 \phi\left(
 \frac{\<c_0p_0(cp_\theta)^\gamma\>}{\<(cp_\theta)^{1+\gamma}\>}
 +\frac{1-c_0}{c\<p_\theta^{1+\gamma}\>}\varepsilon_\theta
 \right)\<(cp_\theta)^{1+\gamma}\>\\
& =
S_\phi(c_0p_0,cp_\theta)+O(\varepsilon_\theta). 
\end{align*}
Suppose that the H\"{o}lder score is strictly proper on the set of non-negative
functions. Then, it is expected that the minimizer of 
$S_\phi(p,cp_\theta)$ is close to $(c,\theta)=(c_0,\theta_0)$.

\subsection{Theoretical Properties of Estimators}
\label{subsec:Theoretical_Properties_Estimators}

Let us consider the optimization of the empirical H\"{o}lder score, 
 \begin{align}
  \label{eqn:Holder-multiplicativemodel-opt}
  \min_{c,\theta} S_\phi(\widetilde{p},cp_\theta),\quad 
  \st\ 0<{c}\leq1,\ \theta\in\Theta, 
 \end{align}
where $\widetilde{p}$ is the empirical probability density of training samples, 
$x_1,\ldots,x_n$. 
%, expressed by using the Dirac's delta function. 
%The parameter space $\Theta$ is assumed to be an open subset 
%in the Euclidean space~$\Rbb^d$. 
We show the relation between the problem \eqref{eqn:Holder-multiplicativemodel-opt} and 
the minimization of the pseudo-spherical score 
 \begin{align}
  \label{eqn:min-pseudospherical}
  \min_{\theta}S_{\mathrm{sphere}}(\widetilde{p},p_\theta),\quad\st\ \theta\in\Theta. 
 \end{align}
 Let $c(\theta)$ be the function
\begin{align}
 \label{eqn:def_c_opt}
 c(\theta)=\frac{\<\widetilde{p}p_{\theta}^\gamma\>}{\<p_{\theta}^{1+\gamma}\>}, 
\end{align}
where
$\<\widetilde{p}p_{\theta}^\gamma\>=\frac{1}{n}\sum_{i=1}^{n}p_{\theta}(x_i)^\gamma$. 
The function $c(\theta)$ connects \eqref{eqn:Holder-multiplicativemodel-opt} 
and \eqref{eqn:min-pseudospherical}. Indeed, we have 
\begin{align*}
 S_\phi(\widetilde{p},cp_\theta)
 =
 \phi\left(
 \frac{\<\widetilde{p}p_\theta^\gamma\>}{c\<p_\theta^{1+\gamma}\>}
 \right) c^{1+\gamma}\<p_\theta^{1+\gamma}\>
 \geq
 -\frac{\<\widetilde{p}p_\theta^\gamma\>^{1+\gamma}}{\<p_\theta^{1+\gamma}\>^\gamma}
 =
 -(-S_{\mathrm{sphere}}(\widetilde{p},p_\theta))^{1+\gamma}, 
\end{align*}
and the equality holds for $c=c(\theta)$. 
Details are presented in the following lemma and theorem. 
The proof is found in Appendix~\ref{appendix:proof_opt_sol}. 
\begin{lemma}
 \label{lemma:pseudo-spherical-Holder}
 For the function $\phi$ in the H\"{o}lder score, suppose $\phi(1)=-1$ and
 $\phi(z)>-z^{1+\gamma}$ for $z\neq1$.  
 For arbitrary positive real number $u$, let us define $\psi_u(z)$ as
 $\psi_u(z)=z^{1+\gamma}\phi(u/z)$ for $z>0$. 
 Suppose that the function $\psi_u(z)$ is strictly decreasing on the open interval
 $(0,u)$. 
 Then, for any fixed parameter $\theta\in\Theta$, the optimal solution of 
 the problem
 \begin{align}
  \label{eqn:unconstraint-Holder-opt}
  \min_{c} S_\phi(\widetilde{p},cp_\theta),\quad  \st\ 0<{c}\leq1
 \end{align}
 is uniquely given as $c=\min\{1,c(\theta)\}$, in which the function $c(\theta)$ is
 defined by \eqref{eqn:def_c_opt}. 
 \end{lemma}

\begin{remark}
 The function $\phi(z)=\gamma-(1+\gamma)z$ that produces the density-power score 
 satisfies the conditions in Lemma~\ref{lemma:pseudo-spherical-Holder}. 
 Let us confirm the condition concerning the function $\psi_u(z)$. 
 For the density-power score, we have $\psi_u(z)=z^{1+\gamma}(\gamma-(1+\gamma)u/z)$, and
 the derivative is $\psi_u'(z)=\gamma(1+\gamma)z^{\gamma-1}(z-u)$. Hence, 
 $\psi_u'(z)<0$ holds for $z\in(0,u)$. 
\end{remark}

\begin{theorem}
 \label{theorem:Holder-min-alg}
 Let $(\widehat{c},\widehat{\theta})$ be an optimal solution of 
 \eqref{eqn:Holder-multiplicativemodel-opt}. 
 In addition to the assumptions in Lemma~\ref{lemma:pseudo-spherical-Holder}, 
 we assume that $c(\theta)$ in \eqref{eqn:def_c_opt} is continuous in the  vicinity of
 $\widehat{\theta}$. 
 If $0<\widehat{c}<1$ holds, the parameter $\widehat{\theta}$ is a local optimal
 solution of \eqref{eqn:min-pseudospherical}.
 Otherwise, 
 the parameter $\widehat{\theta}$ is an optimal solution of the problem
 \begin{align}
  \label{eqn:Holder-opt-c1}
  \min_\theta{}S_\phi(\widetilde{p},p_\theta),\quad\st\,\theta\in\Theta. 
 \end{align}
\end{theorem}

A simple optimization procedure of the problem \eqref{eqn:Holder-multiplicativemodel-opt}
is constructed based on the above theorem. 
Suppose that the assumptions in Theorem~\ref{theorem:Holder-min-alg} holds. 
Moreover, we assume that the problem \eqref{eqn:min-pseudospherical} has the unique local
optimal solution, $\widetilde{\theta}\in\Theta$. 
If $0<c(\widetilde{\theta})\leq1$, the parameter
$(c(\widetilde{\theta}),\widetilde{\theta})$ is an optimal solution of
\eqref{eqn:Holder-multiplicativemodel-opt}. 
Otherwise, solve the problem \eqref{eqn:Holder-opt-c1}, and 
let $\bar{\theta}$ be the optimal solution. 
Then, the point $(c,\theta)=(1,\bar{\theta})$ is an optimal solution of
\eqref{eqn:Holder-multiplicativemodel-opt}. 
Iterative algorithms are available to solve
\eqref{eqn:min-pseudospherical} and \eqref{eqn:Holder-opt-c1}; 
see~\cite{fujisawa08:_robus,a.98:_robus_effic_estim_minim_densit_power_diver} for details. 
When some assumptions in the above argument are violated, 
we use the standard non-linear constrained optimization methods 
such as active set methods. 
Since the constrained inequality $0<c\leq1$ is easy to deal with, 
the non-linear optimization methods will also efficiently work to solve 
\eqref{eqn:Holder-multiplicativemodel-opt}. 

We evaluate the bias of the estimator. 
Let us define $\xi_1=(c_1,\theta_1)$ as an optimal solution of 
\begin{align}
 \label{eqn:min-Holder-score-contam-dist}
 \min_{c,\,\theta}S_{\phi}(p,cp_\theta),\quad\st{}c>0,\ \theta\in\Theta, 
\end{align}
where $p(x)$ is defined by \eqref{eqn:heavy-contamination}, 
and define $\varepsilon_1=\<wp_{\theta_1}^\gamma\>$. 
%Let $\varepsilon_0$ and $\varepsilon_1$ be
%$\varepsilon_0=\<wp_{\theta_0}^\gamma\>$ and 
%$\varepsilon_1=\<wp_{\theta_1}^\gamma\>$, respectively. 
Similarly to Lemma~\ref{lemma:pseudo-spherical-Holder} and 
Theorem~\ref{theorem:Holder-min-alg}, 
the optimal parameter $\xi_1$ does not depend on the
function $\phi$ under a mild assumption. 
\begin{theorem}
\label{theorem:small-bias}
 Suppose 
 that 
 $\xi_1=(c_1,\theta_1)$ is the unique optimal solution of
 \eqref{eqn:min-Holder-score-contam-dist}. 
 Define 
 $f_0(\xi)=S_{\mathrm{power}}(c_0p_0,cp_\theta)$ 
 and 
 $f_1(\xi)=S_{\mathrm{power}}(p,cp_\theta)$ as the function of
 $\xi=(c,\theta)$. 
 For $\xi_1=(c_1,\theta_1)$ and $\varepsilon_1=\<wp_{\theta_1}^\gamma\>$, 
 let $\mathcal{N}$ be a convex set satisfying
 \begin{align*}
  \{\xi\in(0,1]\times\Theta\,|\,f_1(\xi)\leq{}f_1(\xi_1)+(1+\gamma)\varepsilon_1\}
  \subset
  \mathcal{N}. 
 \end{align*}
 Suppose that 
 $f_0(\xi)$ is second order differentiable on $\mathcal{N}$. 
 Let $H_\xi$ be the Hessian matrix of $f_0(\xi)$, and suppose that there exists a
 positive real number $\delta$ 
 such that all eigenvalues of $H_\xi,\,\xi\in\mathcal{N}$ are greater than $\delta$. 
 Then, $\|\xi_1-\xi_0\|=O(\varepsilon_1^{1/2})$ holds. 
\end{theorem}
The proof is found in Appendix~\ref{appendix:small-bias}. 
 % $H_\xi\succ\delta{I}$  % holds for any $\xi\in\mathcal{N}$. 
 % In addition, if the derivative of $\varepsilon_\theta$ is bounded above on an convex open
 % subset of $\Theta$ that includes $\theta_0$ and $\theta_1$, 
 % we have $\|\xi_1-\xi_0\|=O(\varepsilon_{0}^{1/2})$. 

The asymptotic distribution of the estimator based on \eqref{eqn:Holder-multiplicativemodel-opt} 
depends on the parameter $c_0$ of the sample distribution \eqref{eqn:heavy-contamination}. 
For the ratio $c_0$ such that $0<c_0<1$, the standard asymptotic expansion 
is available to derive the asymptotic distribution. 
%leads to the asymptotic normality of the estimator. 
When $c_0=1$, the asymptotic normality will not hold because of the singularity of the
statistical model. 
The asymptotic distribution is, however, obtained by using the
asymptotic expansion under nonstandard
conditions~\cite{self87:_asymp_proper_maxim_likel_estim}. 
The following theorem presents the expression of the asymptotic distribution. 
The matrices $\Sigma_\xi$ and $\Lambda_\theta$, and a small quantity 
$\bar{\varepsilon}$ that appear in the theorem are defined in the proof in 
Appendix~\ref{appendix:asymptotics}. 
\begin{theorem}
 \label{theorem:asymptotic_distribution}
 Let $\xi_0=(c_0,\theta_0)\in(0,1]\times\Theta$ be the target parameter in 
 \eqref{eqn:heavy-contamination}, 
 where $p_0(x)=p_{\theta_0}(x)$ is assumed. 
 Suppose that the conclusion of Theorem~\ref{theorem:small-bias}, i.e., 
 $\|\xi_1-\xi_0\|=O(\varepsilon_1^{1/2})$ holds for $\xi_1=(c_1,\theta_1)$
 that is an optimal solution of \eqref{eqn:min-Holder-score-contam-dist}. 
 An optimal solution of \eqref{eqn:Holder-multiplicativemodel-opt} is 
 denoted as $\widehat{\xi}=(\widehat{c},\widehat{\theta})\in\Rbb^{1+d}$. 
 Let $\widetilde{\theta}$ and $\bar{\theta}$ be the $d$-dimensional optimal solutions of 
 \eqref{eqn:min-pseudospherical} and \eqref{eqn:Holder-opt-c1}, respectively. 
 Then, the following asymptotic properties hold: 

 \begin{enumerate}
  \item Suppose $0<c_0<1$. 
	In addition to the assumptions in Theorem~\ref{theorem:Holder-min-alg}, 
	suppose the regularity conditions such that
	the random vector 
	$\sqrt{n}(c(\widetilde{\theta})-c_1, \widetilde{\theta}-\theta_1)$
	converges in distribution to a $(1+d)$-dimensional multivariate normal distribution with
	the mean zero. 
	Then, the asymptotic distribution of the estimator
	$\widehat{\xi}=(\widehat{c},\widehat{\theta})$ is given as 
	the $d+1$ dimensional normal distribution, i.e., 
	\begin{align*}
	 \sqrt{n}(\widehat{\xi}-\xi_1)\ \cd\ 
	 N_{1+d}(\0,\Sigma_{\xi_0}+O(\bar{\varepsilon}^{1/2})), 
	\end{align*}
	and $\xi_1=\xi_0+O(\bar{\varepsilon}^{1/2})$. 

  \item Suppose $\xi_0=(1,\theta_0)$. 
	In addition to the assumptions in Theorem~\ref{theorem:Holder-min-alg}, 
	suppose the regularity conditions such that
	$\sqrt{n}(c(\widetilde{\theta})-1,\widetilde{\theta}-\theta_0)$
	and 
	$\sqrt{n}(c(\widetilde{\theta})-1,\bar{\theta}-\theta_0)$
	converge in distribution to $(1+d)$-dimensional multivariate normal distributions 
	with the mean zero. 
	Then, the asymptotic distribution of the estimator is expressed as 
	\begin{align*}
	 \sqrt{n}(\widehat{\xi}-\xi_0)
	 \cd{Z}, 
	\end{align*}
	in which $Z=(Z_0,Z_1,\ldots,Z_d)$ is the random variable 
	having the probability density 
	\begin{align*}
	 \phi_{d+1}(z_0,\z_1;\Sigma_{\xi_0})\1[z_0\leq0]
	 +\frac{1}{2}\delta(z_0)
	 \phi_{d}(\z_1;\Lambda_{\theta_0}), 
	\end{align*}
	where $z_0\in\Rbb$ corresponds to $Z_0$ 
	and $\z_1=(z_1,\ldots,z_d)\in\Rbb^d$ corresponds to $(Z_1,\ldots,Z_d)$. 
	Here, $\phi_d(\z;\Sigma)$ denotes the probability density of 
	the distribution $N_d(\0,\Sigma)$, and 
	$\delta(z)$ is the Dirac's delta function. 
	The indicator function $\1[A]$ takes $1$ if $A$ is true, and $0$ otherwise. 
 \end{enumerate} 
\end{theorem}

\begin{remark}
 Some calculation yields that for $\xi=(c,\theta)$, the dependency of $\Sigma_{\xi}$ on 
 $c$ and $\theta$ is given as 
 \begin{align*}
 \Sigma_\xi
 =
 \begin{pmatrix}
  c a_{\theta}-c^2 & \b_{\theta}^T\\ \b_{\theta} & \frac{1}{c}D_{\theta}
 \end{pmatrix},
 \end{align*}
 where $a_{\theta}\in\Rbb,\ \b_{\theta}\in\Rbb^d,\ D_{\theta}\in\Rbb^{d\times{d}}$ are
 quantities that depend only on the parameter $\theta$. 
 When $\gamma$ tends to zero, the vector $\b_\theta$ goes to the zero vector. 
 We omit the concrete expression of the quantities above, 
 since they are somewhat complex. 
 The matrix $\frac{1}{c}D_{\theta}$ is the asymptotic variance of the estimator 
 $\widetilde{\theta}$ based on the pseudo-spherical score. 
 This is proportional to the reciprocal of $c$ that indicates the ratio of samples from the
 target distribution. 
 The same result about the matrix $\frac{1}{c}D_{\theta}$ is presented
 in~\cite{fujisawa08:_robus}. 
\end{remark}

\section{Regression Problems}
\label{sec:Regression_Problems}
Let us consider the application of scoring rules to regression problems. 
In Section~\ref{subsec:Homogeneous_Containation}, the regression problems under
homogeneous contamination is studied. 
In Section~\ref{subsec:Heterogeneous_Contamination}, we deal with heterogeneous
contamination. 
The density-power score and pseudo-spherical score are used to 
derive the estimators for regression problems. 
In~\cite{kanamoriar:_affin_invar_diver_compos_scores_applicy}, 
it is proved that H\"{o}lder score that is available for regression problems is expressed
as a mixture of the density-power score and pseudo-spherical score. 
For simplicity, we focus on the estimators based on the density-power score and
pseudo-spherical score. 
%these two scoring rules for regression problems. 

\subsection{Homogeneous Contamination}
\label{subsec:Homogeneous_Containation}
Let us consider the regression problems based on 
the training samples, $(x_i,y_i),i=1,\ldots,n$ that are i.i.d. samples 
from the joint probability density $p(y|x)q(x)$. 
Under the heavy contamination for the output variable $y$, the conditional density 
$p(y|x)$ is supposed to be expressed as  
\begin{align}
 \label{eqn:conditional-contamination}
 p(y|x)=c_0 p_0(y|x)+(1-c_0)w(y|x), 
\end{align}
where $p_0(y|x)$ is the target conditional density. 
The contamination ratio $1-c_0$ is a constant number that typically lies in the interval 
$[0,1/2)$, i.e., $1/2<c_0\leq{1}$. 
In the above model, the contamination ratio is independent of $x$, and 
such situation is called the homogeneous contamination in this paper. 
The conditional density $w(y|x)$ describes the conditional density of outliers. 
To estimate the target conditional density, we use the parametric model
$p_\theta(y|x)=p(y|x;\theta)$ or its extension, 
$m_{\xi}(y|x)=cp_\theta(y|x)$ with $\xi=(c,\theta)$. 
We assume that the target density is included in the model $p_\theta(y|x)$, i.e., 
$p_0(y|x)$ is expressed as $p_0(y|x)=p_{\theta_0}(y|x)$ for a parameter
$\theta_0\in\Theta$. 

We use the density-power score to estimate the target conditional density. 
Remember that the pseudo-spherical score with the enlarged model does not work to
estimate the contamination ratio. 
Given two functions $f(y|x)$ and $g(y|x)$ having two arguments $x$ and $y$
and a probability density $q(x)$, 
let us define the conditional density-power score as 
\begin{align*}
 S_{\mathrm{power}}(f,g;q)
 =
 \int{}S_{\mathrm{power}}(f(\cdot|x),g(\cdot|x))q(x)dx, 
\end{align*}
where $S_{\mathrm{power}}(f(\cdot|x),g(\cdot|x))$ is the density-power score 
between $f(y|x)$ and $g(y|x)$ as the function of $y$ for a fixed $x$. 
It is straightforward to confirm that the inequality $S_{\mathrm{power}}(f,g;q)\geq{} S_{\mathrm{power}}(f,f;q)$
holds and that the equality $S_{\mathrm{power}}(f,g;q)=S_{\mathrm{power}}(f,f;q)$
leads to $f(y|x)=g(y|x)$ almost everywhere under the measure defined from $q(x)dxdy$. 
By overloading the notation $\<f\>$ of $f(x,y)$ to
representing $\int{}f(x,y)dxdy$, the conditional density-power score is expressed as 
\begin{align*}
 S_{\mathrm{power}}(f,g;q)
 =
 \gamma\<qg^{1+\gamma}\>-(1+\gamma)\<fqg^\gamma\>
\end{align*}

In regression problems based on the samples from $p(y|x)q(x)$, 
we can employ $S_{\mathrm{power}}(p,p_\theta;q)$ or $S_{\mathrm{power}}(p,cp_\theta;q)$ 
as the loss function for statistical inference. 
%We show the empirical approximation of $S_{\mathrm{power}}(p,cp_\theta;q)$. 
Let us define $\widetilde{p}(y|x)\widetilde{q}(x)$ as the empirical probability density 
of the training samples. 
Substituting $\widetilde{p}(y|x)\widetilde{q}(x)$ into $p$ and $q$ in
$S_{\mathrm{power}}(p,cp_\theta;q)$, we obtain the empirical approximation, 
\begin{align*}
 S_{\mathrm{power}}(\widetilde{p},cp_\theta;\widetilde{q}) 
 &=
 -(1+\gamma)\<\widetilde{p}\widetilde{q}(cp_\theta)^\gamma\>
 +\gamma\<\widetilde{q}(cp_\theta)^{1+\gamma}\>\\
 &=
 -\frac{(1+\gamma)c^\gamma}{n}\sum_{i=1}^{n}p_\theta(y_i|x_i)^\gamma{}
 +\frac{\gamma{}c^{1+\gamma}}{n}\sum_{i=1}^{n}\int{}p_\theta(y|x_i)^{1+\gamma}dy. 
\end{align*}
As the sample size tends to infinity, 
the above empirical approximation
converges in probability to 
$S_{\mathrm{power}}(p,cp_\theta;q)$ at each parameter $(c,\theta)$. 
Under the contamination \eqref{eqn:conditional-contamination}, we have 
\begin{align*}
 S_{\mathrm{power}}(p,cp_\theta;q)=
 S_{\mathrm{power}}(c_0p_0,cp_\theta;q)
 -(1+\gamma)(1-c_0)c^\gamma\check{\varepsilon}_\theta, 
\end{align*}
where $\check{\varepsilon}_\theta$ is defined as 
$\check{\varepsilon}_\theta=\<wqp_\theta^\gamma\>$. 
Let $\varepsilon_\theta(x)$ be 
\begin{align}
 \label{eqn:eps-x}
 \varepsilon_\theta(x)=
 \<w(\cdot|x) p_{\theta}(\cdot|x)^{\gamma}\>
 =\int{}w(y|x) p_{\theta}(y|x)^{\gamma}dy, 
\end{align}
then, we have $\check{\varepsilon}_\theta=\int\varepsilon_\theta(x)q(x)dx$. 
In a similar manner to the argument in Section \ref{subsec:Estimation_under_HeavyContainations}, 
since $\varepsilon_\theta(x)$ is expected to be sufficiently small for each $x$, 
so is $\check{\varepsilon}_\theta$ around $\theta=\theta_0$. 
Then, the optimal solution of $\min_{c,\theta}S_{\mathrm{power}}(p,cp_\theta;q)$ will be
close to the optimal solution of $\min_{c,\theta}S_{\mathrm{power}}(c_0p_0,cp_\theta;q)$, 
implying that the minimization of the empirical approximation 
$\min_{c,\theta}S_{\mathrm{power}}(\widetilde{p},cp_\theta;\widetilde{q})$ 
is expected to provide a good estimator of the target parameter $\theta_0$ 
and the ratio $c_0$. 

As shown in Section \ref{subsec:Scoring_Rules_Scaled_Models}, 
the minimization of the conditional density-power score is related to 
the minimization of the pseudo-spherical score. 
In the regression problems, let us define the pseudo-spherical 
score between two conditional probability densities, $p(y|x)$ and $p_\theta(y|x)$ under
the base measure $q(x)$ as  
\begin{align*}
 S_{\mathrm{sphere}}(p,p_\theta;q)=
 -\frac{\<pqp_\theta^\gamma\>}{\<qp_\theta^{1+\gamma}\>^{\gamma/(1+\gamma)}}. 
\end{align*}
Note that the empirical probability density $\widetilde{p}(y|x)\widetilde{q}(x)$ is
directly substituted into $S_{\mathrm{sphere}}(p,p_\theta;q)$. 

Given training samples, the estimator is obtained by solving the problem, 
 \begin{align}
  \label{eqn:conditional-power-div-opt}
 \min_{c,\theta}S_{\mathrm{power}}(\widetilde{p},cp_\theta;\widetilde{q}),\quad
 \st\ 0<c\leq1,\ \theta\in\Theta. 
 \end{align}
%In the same way as the argument in Section~\ref{subsec:Scoring_Rules_Scaled_Models}, 
%the problem \eqref{eqn:conditional-power-div-opt} is related to the optimization of the
%empirical pseudo-spherical score. 
Let us define $c_{\mathrm{reg}}(\theta)$ as
\begin{align*}
 c_{\mathrm{reg}}(\theta)=
 \frac{\<\widetilde{p}\widetilde{q}p_\theta^\gamma\>}{\<\widetilde{q}p_\theta^{1+\gamma}\>}
 =
 \frac{\frac{1}{n}\sum_{i=1}^{n}p_\theta(y_i|x_i)^\gamma}
 {\frac{1}{n}\sum_{i=1}^{n}\int{}p_\theta(y|x_i)^{1+\gamma}dy}. 
\end{align*}
Then, for arbitrary fixed parameter $\theta\in\Theta$, we can verify that 
\begin{align*}
 \min_{c>0}S_{\mathrm{power}}(\widetilde{p},cp_\theta;\widetilde{q})
 =
 S_{\mathrm{power}}(\widetilde{p}, c_{\mathrm{reg}}(\theta)p_\theta;\widetilde{q})
 =
 -(-S_{\mathrm{sphere}}(\widetilde{p},p_\theta;\widetilde{q}))^{1+\gamma}. 
\end{align*}
Hence, in the same way as in Theorem~\ref{theorem:Holder-min-alg}, 
we obtain the theoretical property of the estimator based on 
\eqref{eqn:conditional-power-div-opt}. 
\begin{theorem}
\label{theorem:opt-algorithm-relation-regression}
 Let $(\widehat{c},\widehat{\theta})$ be an optimal solution of
 \eqref{eqn:conditional-power-div-opt}. Suppose that $c_{\mathrm{reg}}(\theta)$ is
 continuous around $\theta=\widehat{\theta}$. 
 If $0<\widehat{c}<1$, the parameter $\widehat{\theta}$ is a local optimal solution
 of the problem, 
 \begin{align}
  \label{eqn:prob-min-cond-sphere}
  \min_\theta{}S_{\mathrm{sphere}}(\widetilde{p},p_\theta;\widetilde{q}),\quad
  \st\ \theta\in\Theta. 
 \end{align}
 Otherwise, the parameter $\widehat{\theta}$ is an optimal solution of 
 \begin{align}
  \label{eqn:prob-min-cond-power-c1}
  \min_\theta{}S_{\mathrm{power}}(\widetilde{p},p_\theta;\widetilde{q}),\quad\st\ \theta\in\Theta. 
 \end{align}
\end{theorem}

We omit the proof, since the proof is almost the same as in
Theorem~\ref{theorem:Holder-min-alg}. 

Based on the above theorem, we present a simple optimization procedure of
the problem~\eqref{eqn:conditional-power-div-opt}. 
We assume that the problem \eqref{eqn:prob-min-cond-sphere} has the unique local optimal
solution, $\widetilde{\theta}\in\Theta$.  
If $0<c_{\mathrm{reg}}(\widetilde{\theta})\leq1$, the parameter
$(c_{\mathrm{reg}}(\widetilde{\theta}),\widetilde{\theta})$ 
is an optimal solution of \eqref{eqn:conditional-power-div-opt}.
Otherwise, for $\theta=\bar{\theta}$ that is an optimal solution of 
\eqref{eqn:prob-min-cond-power-c1}, 
the point $(c,\theta)=(1,\bar{\theta})$ 
is an optimal solution of \eqref{eqn:conditional-power-div-opt}. 
Even if the assumptions in the above argument are violated, 
we can exploit the standard non-linear constrained optimization methods 
such as active set methods. 
Since the constrained inequality $0<c\leq1$ is easy to deal with, 
the non-linear optimization methods will also efficiently work to solve
\eqref{eqn:conditional-power-div-opt}.

\subsection{Location-Scale Models for Heterogeneous Contamination}
\label{subsec:Heterogeneous_Contamination}
We consider the regression problems under the non-constant contamination ratio. 
Suppose that the contaminated conditional probability density of the target $p_0(y|x)$ is
expressed as 
\begin{align*}
 p(y|x)=c_0(x)p_0(y|x)+(1-c_0(x))w(y|x), 
\end{align*}
where $w(y|x)$ denotes the conditional distribution of extreme outliers. 
The contamination ratio $1-c_0(x)$ typically lies in $[0,1/2)$, i.e., 
$1/2<c_0(x)\leq1$ holds at each $x$. We assume $0<c_0(x)\leq1$. 
The situation such that the ratio $c_0$ may depend on the independent variable is 
called heterogeneous contamination. 
To deal with the heterogeneous contamination, 
we assume that the target $p_0(y|x)$ is represented as the location scale model 
\begin{align*}
 p_\theta(y|x)=\frac{1}{\sigma}\,s\left(\frac{y-f_\beta(x)}{\sigma}\right),\ 
 \theta=(\beta,\sigma), \ \sigma>0, 
\end{align*}
where $s(y)$ is a probability density on $\Rbb$ with the mean zero and the unit variance. 
The parameter $\sigma$ denotes the standard deviation, 
and $f_\beta(x)$ with the parameter $\beta$ is the regression function. 
% that is
%the conditional mean of the output $y$ under the conditional density $p_\theta(y|x)$. 
Let us assume that $p_0(y|x)=p_{\theta_0}(y|x)$ holds for a parameter
$\theta_0\in\Theta$. 
The enlarged location scale model is defined as 
$m_{\xi}(y|x)=c{}p_\theta(y|x),\,\xi=(c,\theta)$ for $0<c\leq1$ and $\theta\in\Theta$. 
We show that the constant parameter $c$ efficiently works even under 
heterogeneous contamination. 

The conditional density-power score defined in Section
\ref{subsec:Homogeneous_Containation} is employed. 
The empirical approximation $S_{\mathrm{power}}(\widetilde{p},cp_\theta;\widetilde{q})$ 
converges in probability to $S_{\mathrm{power}}(p,cp_\theta;q)$. 
Let us consider the optimal solution of $\min_{c,\theta}S_{\mathrm{power}}(p,cp_\theta;q)$
under heterogeneous contamination. The direct calculation yields that
\begin{align*}
S_{\mathrm{power}}(p,cp_\theta;q)
&=
S_{\mathrm{power}}(c_0p_0,cp_\theta;q)
 -(1+\gamma)c^\gamma\int(1-c_0(x))q(x)\varepsilon_\theta(x)dx, 
\end{align*}
where $\varepsilon_\theta(x)$ is defined in \eqref{eqn:eps-x}. 
Suppose that $\varepsilon_\theta(x)$ is sufficiently small at each $x$ around
$\theta=\theta_0$. 
Then, the second term of the right-side in the above expression will be negligible, and 
the optimal solution of $S_{\mathrm{power}}(p,cp_\theta;q)$ will be close to the optimal
solution of $S_{\mathrm{power}}(c_0p,cp_\theta;q)$ in which $c_0$ may depend on $x$. 

Let us consider the minimization problem
\begin{align}
 \label{eqn:hetero-reg-c-free}
 \min_{c,\,\theta} S_{\mathrm{power}}(c_0p_0,cp_\theta;q),  \quad c>0,\ \theta\in\Theta. 
\end{align}
We revisit the constraint $c\leq1$ later. 
Using the same idea as in Theorem~\ref{theorem:opt-algorithm-relation-regression}, 
we obtain the inequality 
\begin{align*}
 S_{\mathrm{power}}(c_0p_0,cp_\theta;q)
 \geq
 -\left( \frac{\<c_0p_0qp_\theta^\gamma\>}{\<qp_\theta^{1+\gamma}\>^{\gamma/(1+\gamma)}}
 \right)^{1+\gamma}. 
\end{align*}
The equality holds by setting 
\begin{align*}
 c=\frac{\<c_0p_0qp_\theta^\gamma\>}{\<qp_\theta^{1+\gamma}\>}. 
\end{align*}
In the integral of the location-scale model, the variable change 
$z=(y-f_\beta(x))/\sigma$ produces the equality, 
\begin{align*}
% \<qp_\theta^{1+\gamma}\>
% =
\<p_\theta(\cdot|x)^{1+\gamma}\> =\sigma^{-\gamma}\int{}s(z)^{1+\gamma}dz, 
\end{align*}
i.e., the integral $\int{}p_\theta(y|x)^{1+\gamma}dy$ does not depend on $x$, 
and then, $\<qp_\theta^{1+\gamma}\>=\<p_\theta(\cdot|x)^{1+\gamma}\>$. 
Hence, we obtain 
\begin{align*}
\frac{\<c_0p_0qp_\theta^\gamma\>}{\<qp_\theta^{1+\gamma}\>^{\gamma/(1+\gamma)}}
& =
 \int{}c_0(x)q(x)\frac{\<p_0(\cdot|x)p_\theta(\cdot|x)^\gamma\>}
 {\<p_\theta(\cdot|x)^{1+\gamma}\>^{\gamma/(1+\gamma)}}dx\\
& = 
 -\int{}c_0(x)q(x) S_{\mathrm{sphere}}(p_0(\cdot|x),p_\theta(\cdot|x))dx. 
\end{align*}
Therefore, the optimization of the conditional density-power score is represented as 
\begin{align*}
 \min_{c,\,\theta}S_{\mathrm{power}}(c_0p_0,cp_\theta;q)=
 \min_{\theta}
 -\left( -\int{}c_0(x)q(x) S_{\mathrm{sphere}}(p_0(\cdot|x),p_\theta(\cdot|x))dx\right)^{1+\gamma}.
\end{align*}
The optimal solution of the pseudo-spherical score at each $x$ is given as
$\theta=\theta_0$. 
For the optimal parameter $\theta_0$, the optimal ratio is presented as 
\begin{align*}
 c
 =\frac{\<c_0qp_0^{1+\gamma}\>}{\<qp_{0}^{1+\gamma}\>}
 =\int{}c_0(x)q(x)dx\leq1, 
\end{align*}
where the property of the location-scale models is used in the integral. 

In summary, the optimal solution of the problem  \eqref{eqn:hetero-reg-c-free}
is given by the target mode parameter $\theta=\theta_0$ and 
the expected ratio $c=\int{}c_0(x)q(x)dx$. 
Since the expected ratio is less than or equal to $1$, 
the problem  \eqref{eqn:hetero-reg-c-free} with the additional constraint $c\leq1$ has the
same optimal solution. 
The expected contamination ratio is obtained by $1-\int{}c_0(x)q(x)dx$. 
Therefore, the minimization of the empirical approximate 
$S_{\mathrm{power}}(\widetilde{p},cp_\theta;\widetilde{q})$
will produce an estimator of the target model parameter and the expected contamination ratio 
even under heavy heterogeneous contamination. 

The minimization problem of 
the empirical approximate, $S_{\mathrm{power}}(\widetilde{p},cp_\theta;\widetilde{q})$, 
is common in the homogeneous and heterogeneous situations. 
Hence, Theorem~\ref{theorem:opt-algorithm-relation-regression} 
with the location-scale model also holds in the current situation. 
For the location-scale model, 
the integral in $\<\widetilde{q}p_\theta^{1+\gamma}\>$ 
is expressed as $\sigma^{-\gamma}\int{}s(z)^{1+\gamma}dz$. 
Once the integral of $s(z)^{1+\gamma}$ is computed, 
any additional integral is not required in the process of the optimization. 
This is a computational advantage of the location-scale model.

\section{Numerical Experiments}
\label{sec:Numerical_Experiments}
We conducted numerical experiments to evaluate the statistical properties
of robust estimators including the preceding technical developments. 
First, synthetic datasets for density estimation problems and regression problems were
employed. 
Then, benchmark datasets were used to compare robust estimators for regression problems. 
We borrowed the setup of regression problems from \cite{yu12:_polyn_form_robus_regres}.

\subsection{Synthetic data}
\label{subsec:Synthetic}

First, we show illustrative examples of robust estimation. 
\paragraph{Density Estimation:}
The training samples $x_1,\ldots,x_{n}\in\Rbb^2$ were drawn from the two-dimensional
standard normal distribution $N_2(\0,I)$, where $\0$ is the zero vector and $I$ is the identity matrix. 
To seed the outliers, $20\%$ of the training samples were randomly chosen and 
their values were replaced with the samples each component of which was generated from the
normal distribution $N(10,10^2)$. The sample size was set to $n=50$. 
Figure~\ref{fig:plot-normal-robustest} depicts the scatter plot of the 
observations including outliers. 
The statistical model $p_\theta(x)$ is the full-model of the two-dimensional normal
distribution $N_2({\bm\mu},\Sigma)$, i.e., 
the five dimensional parameter $\theta$ consists of the mean vector and the
variance-covariance matrix. 
The estimated parameter based on the maximum likelihood estimator was given as 
\begin{align*}
 \widehat{\bm\mu}_{\mathrm{MLE}}=
 \begin{pmatrix}
  2.70\\ 1.86
 \end{pmatrix},\quad 
 \widehat{\Sigma}_{\mathrm{MLE}}
 =
 \begin{pmatrix} 
 39.40 &20.76\\
 20.76 &20.28
\end{pmatrix}. 
\end{align*}
As the robust estimator, we employed the density-power score $S_{\mathrm{power}}$ with
$\gamma=0.1$ and the enlarged model $cp_\theta(x)$. 
Then, the estimated parameter of the target density $N_2(\0,I)$ was
\begin{align*}
 \widehat{\bm\mu}=\begin{pmatrix}
		   0.05\\ 0.11\end{pmatrix},\quad 
 \widehat{\Sigma} =
 \begin{pmatrix}    0.91& -0.03 \\ -0.03&  0.85\end{pmatrix}. 
\end{align*}
In addition, the proposed method provided the estimator of the contamination ratio
$1-\widehat{c}$. By picking up $50(1-\widehat{c})$ samples in ascending order of
the estimated values $p_{\widehat{\theta}}(x_i)$, 
one can identify the outliers. 
In this example, the estimated contamination ratio was $0.198$, and 
the detected outliers are indicated as the triangle points in
Figure~\ref{fig:plot-normal-robustest}. 

 \begin{figure}
 \begin{center}
   \includegraphics[scale=0.4]{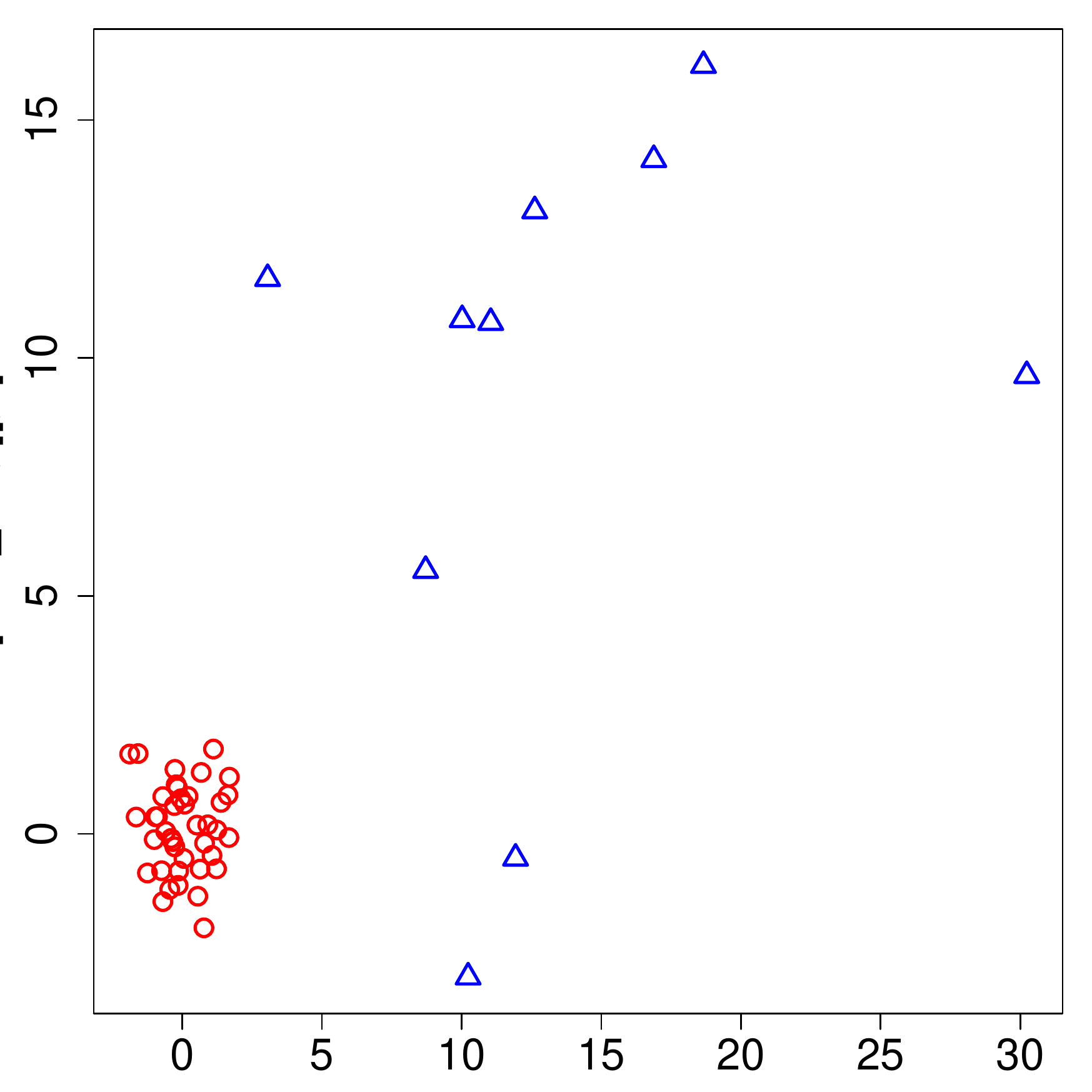}
 \end{center} 
  \caption{The scatter plot of training samples. 
  The triangle points indicate the detected outliers by our methods. }
  \label{fig:plot-normal-robustest}
\end{figure}

\paragraph{Regression:}
%For regression problems, the setup is the following. 
Let us consider the simple linear regression problem. 
The independent variable $x\in\Rbb$ was drawn from the standard normal distribution
$N(0,1)$, 
and the target density $p_0(y|x)$ was defined from the regression function
$y=1+10x+\varepsilon$, 
where the noise $\varepsilon$ is generated from $N(0,1)$. 
As the outlier $(x,y)$, 
$x$ was drawn from $N(1,0.8^2)$ and $y$ was 
the absolute value of the random variable drawn from $N(0,70^2)$. 
The left panel of Figure~\ref{fig:plot-robustest} depicts the scatter plot of 
the observations including outliers. 
The sample size was $50$, and the expected contamination ratio was set to $1-c_0=0.3$. 
The right panel presents the estimated regression functions based on the least square
estimator and the proposed method using the density-power score with $\gamma=0.1$. 
Our approach produced a reasonable result, while the least square estimator was
significantly affected by the outliers.  
By picking up $50(1-\widehat{c})$ samples in ascending order of the estimated value of the
conditional probability density, $p_{\widehat{\theta}}(y_i|x_i)$, 
one can identify the outliers. % from all the observations. 
The estimated contamination ratio was $0.265$, and 
the triangle points denote the detected outliers. 
\begin{figure}
 \begin{center}
  \begin{minipage}{.45\textwidth}
   \begin{center}
    \scalebox{0.4}{\includegraphics{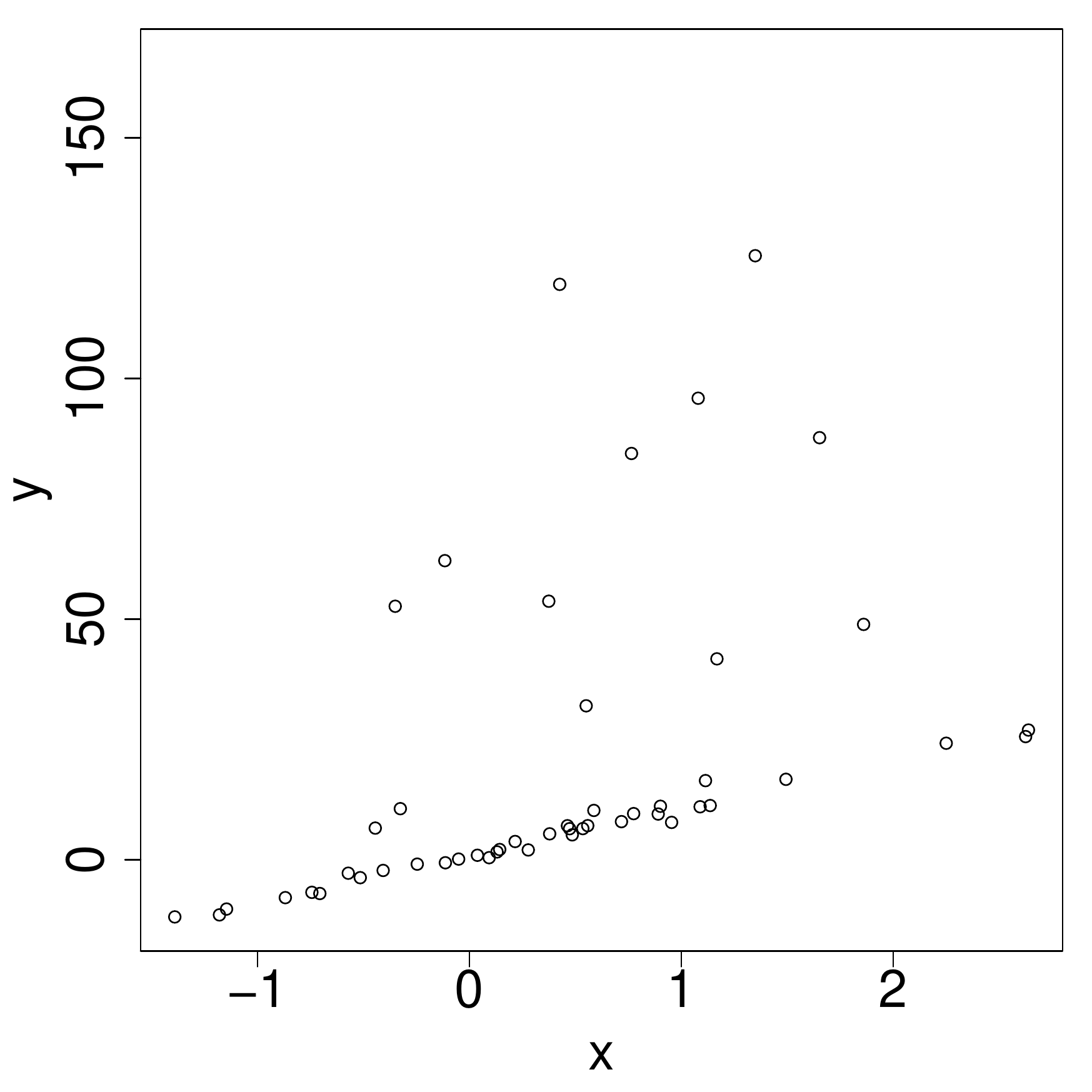}}
   \end{center}
  \end{minipage}
   \hspace*{10mm}
  \begin{minipage}{.45\textwidth}
   \begin{center}
    \scalebox{0.4}{\includegraphics{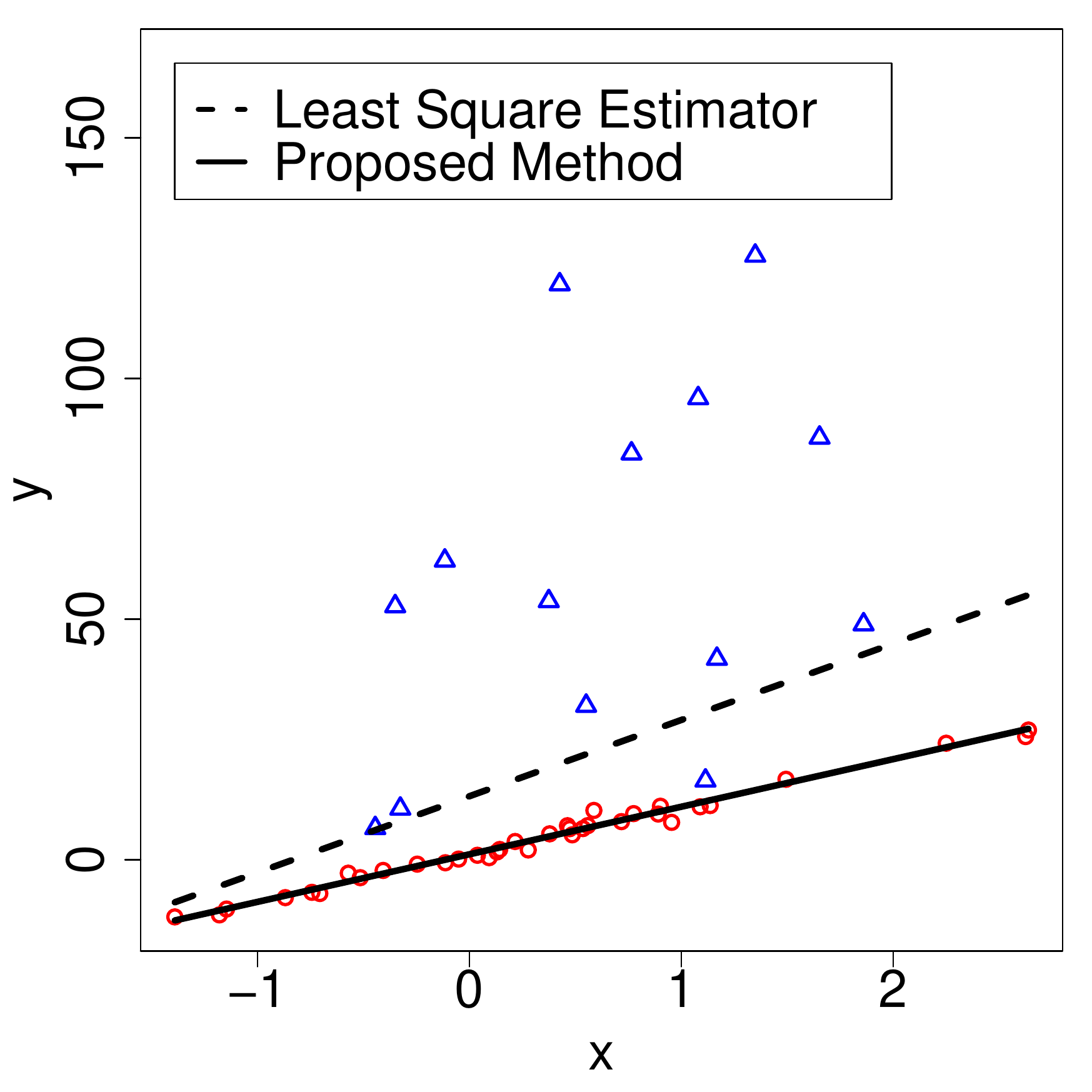}}
   \end{center}
  \end{minipage}
  \caption{Left panel: the scatter plot of training samples for regression estimation. 
  Right panel: The solid line is the estimated regression function by our methods, 
  and the broken line is the estimation result of the least square estimator. 
  The triangle points are identified as the outliers by our methods. }
 \label{fig:plot-robustest}
 \end{center}
\end{figure}

Next, we present numerical experiments of linear regression problems under heavy contamination. 
The problem setup is similar to the setup in~\cite{yu12:_polyn_form_robus_regres}. 
For $x\in\Rbb^d, d=5$ and $y\in\Rbb$, the target density $p_0(y|x)$ was defined from
the regression function $y=x^T\theta_0+\varepsilon$, where the target parameter 
$\theta_0$ was generated from the multivariate normal distribution $N_d(\0,I)$. 
The distribution of the noise $\varepsilon$ was the normal distribution $N(0,1/4)$, and 
the independent variable $x$ was drawn from the uniform distribution on $[0,1]^d$. 
The estimation accuracy was evaluated on $1000$ test points that were drawn from 
the joint probability of $(x,y)$ in the above. 

Let us consider two setups for contamination. 
In the first setup, each dependent variable $y_i$ was re-sampled as the outlier from
$N(0,10^8)$ with the contamination probability $1-c_0$, 
while the independent variable was not changed. 
In the second, both $x$ and $y$ were resampled from $N(0,10^4)$ and $N(0,10^8)$, 
respectively. 
The estimators using enlarged models are designed to deal with heavy contamination in the
first setup. 
We present that the proposed methods efficiently work even in the second setup. 

In the regression problems, the following methods were compared: 
least square method (L2), 
median regression estimator base on $L_1$-loss (L1), 
robust estimator using Huber loss (Huber)~\cite{huber64:_robus},
least trimmed square method (LTS)~\cite{rousseeuw06:_comput_lts_regres_large_data_sets},
robust estimator using the bounded Geman-McClure loss (GemMc)~\cite{black96}, 
robust MM-estimator (MM-est)~\cite[Chap.~5]{maronna06:_robus_statis}, 
and the proposed method using the density-power score with enlarged model ($S_{\mathrm{power}}$). 
The LTS method requires an estimate of the contamination ratio. 
In our experiments, the true ratio $1-c_0$ was fed to the LTS method. 
In the present setup, the linear regression model includes the intercept, 
while the regression model used in \cite{yu12:_polyn_form_robus_regres} did not have the intercept. 
The model $p_\theta(y|x)$ with the parameter
$\theta=({\bm\beta}_0,\beta_1,\sigma)$ was defined from 
$y=\beta_1+{\bm\beta}_0^T\x+\varepsilon,\,\varepsilon\sim{}N(0,\sigma^2)$, and 
the enlarged model was given as $cp_\theta(y|x)$. 

For each estimator, we computed the averaged root mean square errors (RMSE) over 100
iterations. 
The contamination ratio estimated by using the proposed methods is also presented. 
The upper part of Table~\ref{tbl:normaldist-synthetic}
reports the numerical results of the first setup, i.e.,
contamination only for the dependent variable. 
When the samples were not contaminated, 
all estimators efficiently worked as shown in the left column of the table. 
Indeed, the all RMSEs were close to optimal value $1/2$, i.e., 
the standard deviation of the noise $\varepsilon$. 
This result is almost the same as that in \cite{yu12:_polyn_form_robus_regres}. 
As shown in the middle and right columns, 
the least square method and Huber estimator tended to be affected by outliers. 
The lower part of Table~\ref{tbl:normaldist-synthetic} reports the results of the second setup. 
In addition to L2 and Huber, the L1-estimator was degraded by outliers. 
Even under heavy contamination, GemMc, MM-est and the proposed method performed well. 
We also found that the estimator $S_{\mathrm{power}}$ was useful for the estimation of the
contamination ratio even under the second setup. In this experiments, the choice of
$\gamma$ in the density-power score did not significantly affect the estimation 
accuracy.

\begin{table}[t]
\centering
\begin{tabular}{c|ccc}
                                    & \multicolumn{3}{c}{Outlier Probability for variable $y$} \\ \cline{2-4}
  Methods                           & {\ $1-c_0=0.0$\ }             & $1-c_0=0.2$              &   {$1-c_0=0.4$}      \\ \hline
   L2                               & 0.51 $\pm$  0.01 & 1093.9 $\pm$ 358.97&1528.91 $\pm$ 454.99   \\
   L1                               & 0.52 $\pm$  0.01 & 0.54   $\pm$ 0.02  &  0.58  $\pm$ 0.05     \\
   Huber                            & 0.52 $\pm$  0.01 & 1.40   $\pm$ 0.21  &621.76  $\pm$ 335.89   \\
   LTS                              & 0.52 $\pm$  0.01 & 0.52   $\pm$ 0.01  & 15.64  $\pm$ 19.07    \\
   GemMc                            & 0.52 $\pm$  0.01 & 0.52   $\pm$ 0.02  &  0.54  $\pm$ 0.02     \\
   MM-est                           & 0.52 $\pm$  0.01 & 0.52   $\pm$ 0.01  &  0.53  $\pm$ 0.02     \\
$S_{\mathrm{power}}\ (\gamma=0.1)$  & 0.52 $\pm$  0.01 & 0.52   $\pm$ 0.01  &  0.53  $\pm$ 0.02     \\
$S_{\mathrm{power}}\ (\gamma=0.5)$  & 0.52 $\pm$  0.01 & 0.52   $\pm$ 0.01  &  0.54  $\pm$ 0.02     \\
$S_{\mathrm{power}}\ (\gamma=1.0)$  & 0.53 $\pm$  0.02 & 0.54   $\pm$ 0.03  &  0.56  $\pm$ 0.03     \\ \hline
$1-\widehat{c}\ (\gamma=0.1)$         & 0.00 $\pm$  0.00 & 0.17   $\pm$ 0.07  &  0.36  $\pm$ 0.13     \\
$1-\widehat{c}\ (\gamma=0.5)$         & 0.01 $\pm$  0.01 & 0.19   $\pm$ 0.05  &  0.40  $\pm$ 0.04     \\
$1-\widehat{c}\ (\gamma=1.0)$         & 0.06 $\pm$  0.10 & 0.25   $\pm$ 0.09  &  0.42  $\pm$ 0.15     \\
\end{tabular}
%\end{table}
%\begin{table}[t]
\centering
\vspace*{5mm}

\begin{tabular}{c|ccc}
                                    & \multicolumn{3}{c}{Outlier Probability for variables $(\x,y)$} \\ \cline{2-4}
  Methods                           &   {\ $1-c_0=0.0$\ }       & $1-c_0=0.2$         & $1-c_0=0.4$              \\ \hline
   L2                               &  0.52 $\pm$  0.01 &313.76 $\pm$ 232.81         & 532.41 $\pm$ 353.29      \\
   L1                               &  0.52 $\pm$  0.01 &\!\!\!\!\! 18.80 $\pm$ 6.73 & \!\!\!13.71  $\pm$ 5.21  \\
   Huber                            &  0.52 $\pm$  0.01 &\!\!\!\!\! 18.98 $\pm$ 6.96 & 60.11  $\pm$ 64.27       \\
   LTS                              &  0.52 $\pm$  0.01 &1.17   $\pm$ 0.79 	     & 1.23   $\pm$ 0.63        \\
   GemMc                            &  0.52 $\pm$  0.01 &0.53   $\pm$ 0.02 	     & 0.54   $\pm$ 0.02        \\
   MM-est                           &  0.52 $\pm$  0.01 &0.52   $\pm$ 0.01 	     & 0.54   $\pm$ 0.02        \\
$S_{\mathrm{power}}\ (\gamma=0.1)$  &  0.52 $\pm$  0.01 &0.52   $\pm$ 0.01	     & 0.53   $\pm$ 0.02        \\
$S_{\mathrm{power}}\ (\gamma=0.5)$  &  0.52 $\pm$  0.01 &0.52   $\pm$ 0.02	     & 0.54   $\pm$ 0.02        \\
$S_{\mathrm{power}}\ (\gamma=1.0)$  &  0.53 $\pm$  0.02 &0.54   $\pm$ 0.03           & 0.56   $\pm$ 0.05        \\ \hline
$1-\widehat{c}\ (\gamma=0.1)$         &  0.00 $\pm$  0.00 & 0.17 $\pm$ 0.07            & 0.33   $\pm$ 0.15        \\
$1-\widehat{c}\ (\gamma=0.5)$         &  0.01 $\pm$  0.01 & 0.20 $\pm$ 0.04            & 0.40   $\pm$ 0.04        \\
$1-\widehat{c}\ (\gamma=1.0)$         &  0.05 $\pm$  0.06 & 0.26 $\pm$ 0.10            & 0.41   $\pm$ 0.16        \\
\end{tabular}
 \caption{
 Top table shows the results in the case that only the dependent variable $y$ is incurred
 by outliers. Bottom table shows the results in the case that both independent and
 dependent variables $(x,y)$ are contaminated. 
 In the synthetic regression problems, 
 RMSEs on 10000 clean test samples are computed. 
 The training sample size is 100, and an observation consists of the pair of the
 $5$-dimensional independent variable and a dependent variable. 
 The contamination ratio $1-c_0$ is set to $0.0$ (clean training data), $0.2$ or
 $0.4$. 
 For each estimator, we employed the linear regression model with the intercept. 
 The averaged RMSE over 100 iterations is computed, and 
 also the estimated contamination ratio is presented. 
 }
 \label{tbl:normaldist-synthetic}
\end{table}

\subsection{Benchmark data}
\label{subsec:Benchmark_data}
We used four benchmark datasets taken from the StatLib repository and DELVE: 
cal-housing, abalone, pumadyn-32fh, and bank-8fh. 
These were the same as the datasets used in \cite{yu12:_polyn_form_robus_regres}. 
Cal-housing dataset has 8 features and one dependent variable (median House Value). 
Abalon dataset has originally 8 features and one output (rings). However, one discrete
feature, ``Gender or Infant'', is removed, and we use 7 features and one 
dependent variable. 
Pumadyn-32fh has 32 features and one output variable (ang acceleration of joint 6). 
Bank-8fh has 8 features and one output variable (rejection rate). 
The dependent variable of bank-8fh dataset denotes the probability, and hence, 
the logistic regression would be appropriate to analyze bank-8fh. 
However, we dealt with the rejection rate just as a real number 
in order to investigate the robustness property of the regression estimators. 

For each dataset, 100 training samples and 1000 test samples were randomly selected. 
Let us consider two kinds of contamination, i.e., contamination of only dependent
variable ($y$-contamination),  
and that of both independent and dependent variables ($(x,y)$-contamination). 
To seed outliers, some amount of training samples were randomly chosen, 
and their $y$ values were multiplied by 10000 in the first setup. 
In the second setup, $x$ values were also multiplied by 100. 
The contamination ratio was set to $1-c_0=0.05, 0.2$ or $0.4$, while 
only the case of $1-c_0=0.05$ was examined in \cite{yu12:_polyn_form_robus_regres}. 

For the model fitting, we employed the linear regression model with the intercept. 
In addition, the normal distribution was assumed for the conditional probability model
$p_\theta(y|x)$. 
In~\cite{yu12:_polyn_form_robus_regres}, the regularization technique was used. 
In the numerical experiments of this article, we did not use the regularization, 
since the regression model used in the experiments was rather simple. 
Again, the true contamination ratio $1-c_0$ was used in the LTS estimator. 

Table~\ref{tbl:benchmark-y_out} (resp. Table~\ref{tbl:benchmark-xy_out}) reports the
RMSE on the test samples under the setup of $y$-contamination
(resp. $(x,y)$-contamination). 
As shown in \cite{maronna06:_robus_statis,yu12:_polyn_form_robus_regres}, 
any estimator based on minimizing a convex loss such as L2, L1, and Huber 
was sensitive to even small amount of outliers. 
Under the heavy contamination, also the LTS estimator was degraded by the outliers. 
Other estimators, GemMc, MM-est and $S_{\mathrm{power}}$ were not degraded even under 
heavy contamination. 
In both $y$-contamination and $(x,y)$-contamination, 
$S_{\mathrm{power}}$ with $\gamma=0.1$ efficiently performed for the estimation 
of the model parameter $\theta$ and the contamination ratio $1-c_0$. 
Also, other estimators based on non-convex losses such as GemMc, MM-est and
$S_{\mathrm{power}}$ with $\gamma=0.5$ and $1.0$ provided rather stable results. 
In Pumadyn-32fh dataset, the MM-est performed worse. 
In our experiments, the MM-est got sensitive for fairly high-dimensional data. 
When the estimator was trapped in local minima, the estimation accuracy was not high. 
%On the other hand, 
The estimator $S_{\mathrm{power}}$ with a small $\gamma$ was expected to have the unique
local minima. Thus, the problematic local minima would be avoided. 
In practice, $S_{\mathrm{power}}$ with $\gamma=0.1$ provided an accurate estimator.

\begin{table}[t]\small
\centering
\begin{tabular}{c|cccc}
                                    & \multicolumn{4}{c}{Outlier Probability for $y$: $5\%$} \\ \cline{2-5}
  Methods                           & cal-housing            & abalone              &  pumadyn-32fh  & bank-8fh  \\ \hline
   L2                                &13963.18 $\pm$  8868.55  &8739.79 $\pm$  3095.99    &38.125 $\pm$  12.825 &170.797$\pm$ 86.748  \\		    
   L1                                &964.23   $\pm$  6157.65  &2.42    $\pm$  0.17 	    & 0.028 $\pm$  0.002  & 0.079 $\pm$ 0.005  \\		    
   Huber                             &969.03   $\pm$  6201.87  &2.69    $\pm$  0.34 	    & 0.028 $\pm$  0.002  & 0.081 $\pm$ 0.006  \\		    
   LTS                               &0.43     $\pm$  0.12     &2.38    $\pm$  0.15 	    & 0.025 $\pm$  0.001  & 0.077 $\pm$ 0.004  \\		    
   GemMc                             &0.42     $\pm$  0.10     &2.65  $\pm$  0.20	    & 0.025 $\pm$  0.001  & 0.077 $\pm$ 0.004  \\		    
   MM-est                            &0.46     $\pm$  0.18     &2.42    $\pm$  0.15 	    & 0.026 $\pm$  0.002  & 0.078 $\pm$ 0.004  \\		    
$S_{\mathrm{power}}\ (\gamma=0.1)$   &0.42     $\pm$  0.11     &2.36    $\pm$  0.14 	    & 0.025 $\pm$  0.001  & 0.076 $\pm$ 0.003  \\		    
$S_{\mathrm{power}}\ (\gamma=0.5)$   &0.42     $\pm$  0.08     &2.46    $\pm$  0.15 	    & 0.033 $\pm$  0.004  & 0.079 $\pm$ 0.004  \\		    
$S_{\mathrm{power}}\ (\gamma=1.0)$   &0.46     $\pm$  0.16     &2.56    $\pm$  0.18 	    & 0.032 $\pm$  0.003  & 0.083 $\pm$ 0.007  \\ \hline	    
$1-\widehat{c}\ (\gamma=0.1)$&0.02     $\pm$  0.00     &0.05  $\pm$  0.00	    & 0.050 $\pm$  0.005  & 0.051 $\pm$ 0.002  \\		    
$1-\widehat{c}\ (\gamma=0.5)$&0.00     $\pm$  0.00     &0.13  $\pm$  0.04	    & 0.009 $\pm$  0.089  & 0.097 $\pm$ 0.030  \\		    
$1-\widehat{c}\ (\gamma=1.0)$&0.01     $\pm$  0.02     &0.25    $\pm$  0.16       & 0.000 $\pm$  0.000  & 0.166 $\pm$ 0.142  \\             
\end{tabular}
\vspace*{5mm}

\begin{tabular}{c|cccc}
                                    & \multicolumn{4}{c}{Outlier Probability for $y$: $20\%$} \\ \cline{2-5}
  Methods                           & cal-housing             & abalone             &pumadyn-32fh   & bank-8fh  \\ \hline
   L2                                &33437.96 $\pm$ 13734.96 & 24434.80  $\pm$  3626.46   & 86.000 $\pm$  17.746  &474.450 $\pm$  128.359  \\		    
   L1                                &3131.27  $\pm$ 8743.01  &\    16.32 $\pm$   136.38   &  0.313 $\pm$   2.731  & 0.085  $\pm$  0.008    \\		    
   Huber                             &5864.10  $\pm$ 12788.48 &\   404.18 $\pm$  2136.62   & 41.666 $\pm$  21.285  & 0.740  $\pm$  0.299    \\		    
   LTS                               &13667.6  $\pm$ 39274.54 &\   158.53 $\pm$  1084.07   & 12.867 $\pm$  7.737  & 0.754  $\pm$  1.263    \\		    
   GemMc                             & 0.43    $\pm$ 0.12     & 2.67    $\pm$   0.23     &  0.027 $\pm$   0.004  & 0.078  $\pm$  0.004    \\		    
   MM-est                            & 0.47    $\pm$ 0.19     & 2.40    $\pm$   0.17     &  0.027 $\pm$   0.002  & 0.078  $\pm$  0.004    \\		    
$S_{\mathrm{power}}\ (\gamma=0.1)$   & 0.43    $\pm$ 0.12     & 2.40    $\pm$   0.18     &  0.027 $\pm$   0.002  & 0.078  $\pm$  0.004    \\		    
$S_{\mathrm{power}}\ (\gamma=0.5)$   & 0.44    $\pm$ 0.11     & 2.47    $\pm$   0.18     &  0.035 $\pm$   0.005  & 0.080  $\pm$  0.004    \\		    
$S_{\mathrm{power}}\ (\gamma=1.0)$   & 0.50    $\pm$ 0.25     & 2.59    $\pm$   0.23     &  0.034 $\pm$   0.004  & 0.082  $\pm$  0.006    \\ \hline	    
$1-\widehat{c}\ (\gamma=0.1)$& 0.17    $\pm$ 0.00     & 0.20    $\pm$   0.00     &  0.171 $\pm$   0.072  & 0.176  $\pm$  0.065    \\		    
$1-\widehat{c}\ (\gamma=0.5)$& 0.10    $\pm$ 0.01     & 0.27    $\pm$   0.03     &  0.000 $\pm$   0.000  & 0.230  $\pm$  0.076    \\		    
$1-\widehat{c}\ (\gamma=1.0)$& 0.04    $\pm$ 0.05     & 0.26      $\pm$   0.21     &  0.000 $\pm$   0.000  & 0.160  $\pm$  0.197    \\             
\end{tabular}
\vspace*{5mm}

\begin{tabular}{c|cccc}
                                    & \multicolumn{4}{c}{Outlier Probability for $y$: $40\%$} \\ \cline{2-5}
  Methods                           & cal-housing              & abalone                     &pumadyn-32fh  & bank-8fh  \\ \hline
   L2                                &56697.75 $\pm$ 10554.54  & 44463.78 $\pm$   4040.32    &\!\!127.70 $\pm$ 18.67 & 915.905 $\pm$ 177.846 \\		    
   L1                                &31030.68 $\pm$ 30800.29  & 11495.54 $\pm$  16592.47    &  76.54 $\pm$ 32.96 & 124.731 $\pm$ 251.276 \\		    
   Huber                             &56757.25 $\pm$ 10565.04  & 42517.89 $\pm$  4353.52     &\!\!113.72 $\pm$ 20.57 & 689.616 $\pm$ 266.240  \\		    
   LTS                               &86798.76 $\pm$ 114753.60 & 9856.11  $\pm$  4539.19     &  62.15 $\pm$ 34.24 & 14.010  $\pm$  9.477   \\		    
   GemMc                             &0.48     $\pm$ 0.22      & 2.78     $\pm$  0.37        &   0.04 $\pm$ 0.02  & 0.080   $\pm$ 0.005   \\		    
   MM-est                            &0.51     $\pm$ 0.25      & 2.49     $\pm$  0.26        &  86.23 $\pm$ 33.46 & 0.080   $\pm$ 0.005   \\		    
$S_{\mathrm{power}}\ (\gamma=0.1)$   &0.49     $\pm$ 0.24      & 2.47     $\pm$  0.25        &   0.03 $\pm$ 0.00  & 0.080   $\pm$ 0.005   \\		    
$S_{\mathrm{power}}\ (\gamma=0.5)$   &0.52     $\pm$ 0.29      & 2.56     $\pm$  0.23        &  \ 5.98 $\pm$ 24.11 & 0.081   $\pm$ 0.005   \\		    
$S_{\mathrm{power}}\ (\gamma=1.0)$   &0.54     $\pm$ 0.31      & 2.58     $\pm$  0.23        &  14.08 $\pm$ 40.78 & 0.081   $\pm$ 0.005   \\ \hline	    
$1-\widehat{c}\ (\gamma=0.1)$& 0.37    $\pm$ 0.08      & 0.38     $\pm$  0.09        &   0.28 $\pm$ 0.16  & 0.192   $\pm$ 0.201   \\		    
$1-\widehat{c}\ (\gamma=0.5)$& 0.32    $\pm$ 0.06      & 0.45     $\pm$  0.06        &   0.00 $\pm$ 0.00  & 0.123   $\pm$ 0.198   \\		    
$1-\widehat{c}\ (\gamma=1.0)$& 0.10    $\pm$ 0.15      & 0.23     $\pm$  0.27        &   0.00 $\pm$ 0.00  & 0.010   $\pm$ 0.074   \\             
\end{tabular}
 \caption{
 The numerical results on benchmark datasets are presented. 
 The training samples size is 100, and the contamination ratio $1-c_0$ is set to 
 $0.05, 0.2$ or 
 $0.4$. Only the dependent variable $y$ is contaminated by the outliers. 
 For each estimator, we employed the linear regression model with the intercept, and computed
 the averaged RMSE over 100 iterations. Also the estimated contamination ratio is presented. 
}
 \label{tbl:benchmark-y_out}
\end{table}

\begin{table}[t]\small
\centering
\begin{tabular}{c|cccc}
                                    & \multicolumn{4}{c}{Outlier Probability for $(\x,y)$: $5\%$} \\ \cline{2-5}
  Methods                           & cal-housing            & abalone              &  pumadyn-32fh  & bank-8fh  \\ \hline
   L2                               & 241.76 $\pm$ 243.42&451.65 $\pm$ 283.76&1.152 $\pm$ 0.407&12.763 $\pm$ 4.752    \\
   L1                               & 262.60 $\pm$ 269.21&393.75 $\pm$ 209.54&1.246 $\pm$ 0.457&12.679 $\pm$ 4.773    \\
   Huber                            & 245.04 $\pm$ 256.31&402.75 $\pm$ 216.20&1.157 $\pm$ 0.412&12.635 $\pm$ 4.634    \\
   LTS                              & 0.43   $\pm$ 0.11  & 2.39  $\pm$ 0.18  &0.063 $\pm$ 0.045&0.078  $\pm$ 0.003    \\
   GemMc                            & 0.43   $\pm$ 0.11  & 2.65  $\pm$ 0.24  &0.026 $\pm$ 0.003&0.077  $\pm$ 0.003    \\
   MM-est                           & 0.47   $\pm$ 0.17  & 2.42  $\pm$ 0.21  &0.026 $\pm$ 0.002&0.078  $\pm$ 0.003    \\
$S_{\mathrm{power}}\ (\gamma=0.1)$  & 0.43   $\pm$ 0.11  & 2.38  $\pm$ 0.18  &0.026 $\pm$ 0.002&0.077  $\pm$ 0.003    \\
$S_{\mathrm{power}}\ (\gamma=0.5)$  & 0.45   $\pm$ 0.15  & 2.48  $\pm$ 0.23  &0.034 $\pm$ 0.004&0.079  $\pm$ 0.004    \\
$S_{\mathrm{power}}\ (\gamma=1.0)$  & 0.52   $\pm$ 0.27  & 2.62  $\pm$ 0.26  &0.032 $\pm$ 0.004&0.084  $\pm$ 0.006    \\ \hline
$1-\widehat{c}\ (\gamma=0.1)$         & 0.05   $\pm$ 0.00  & 0.05  $\pm$ 0.00  &0.05  $\pm$ 0.01 &0.05   $\pm$ 0.00     \\
$1-\widehat{c}\ (\gamma=0.5)$         & 0.08   $\pm$ 0.02  & 0.12  $\pm$ 0.03  &0.02  $\pm$ 0.12 &0.10   $\pm$ 0.03     \\
$1-\widehat{c}\ (\gamma=1.0)$         & 0.17   $\pm$ 0.13  & 0.24  $\pm$ 0.12  &0.00  $\pm$ 0.00 &0.18   $\pm$ 0.15     \\
\end{tabular}
\vspace*{5mm}

\begin{tabular}{c|cccc}
                                    & \multicolumn{4}{c}{Outlier Probability for $(\x,y)$: $20\%$} \\ \cline{2-5}
  Methods                            & cal-housing         & abalone             &pumadyn-32fh   & bank-8fh  \\ \hline
   L2                                & 340.83 $\pm$  281.14 &363.19 $\pm$ 117.88  &3.16 $\pm$ 0.73&15.048 $\pm$ 2.981  \\
   L1                                & 313.20 $\pm$  288.40 &325.39 $\pm$ 118.34  &3.29 $\pm$ 0.76&14.487 $\pm$ 3.260  \\
   Huber                             & 311.71 $\pm$  282.82 &326.34 $\pm$ 120.79  &3.17 $\pm$ 0.73&14.281 $\pm$ 3.143  \\
   LTS                               & 299.28 $\pm$  334.20 &86.27  $\pm$ 17.50   &0.32 $\pm$ 0.11&0.288  $\pm$ 0.174  \\
   GemMc                             & 0.45   $\pm$  0.17   &2.67 $\pm$ 0.24    &0.03 $\pm$ 0.01&0.078  $\pm$ 0.004  \\
   MM-est                            & 0.52   $\pm$  0.36   &2.40 $\pm$ 0.15    &0.08 $\pm$ 0.08&0.078  $\pm$ 0.004  \\
$S_{\mathrm{power}}\ (\gamma=0.1)$   & 0.45   $\pm$  0.18   &2.39 $\pm$ 0.15    &0.03 $\pm$ 0.01&0.078  $\pm$ 0.004  \\
$S_{\mathrm{power}}\ (\gamma=0.5)$   & 0.48   $\pm$  0.24   &2.50 $\pm$ 0.22    &0.07 $\pm$ 0.08&0.080  $\pm$ 0.005  \\
$S_{\mathrm{power}}\ (\gamma=1.0)$   & 0.51   $\pm$  0.25   &2.60 $\pm$ 0.25    &0.05 $\pm$ 0.03&0.083  $\pm$ 0.007  \\ \hline
$1-\widehat{c}\ (\gamma=0.1)$& 0.17   $\pm$  0.00   &0.20 $\pm$ 0.00    &0.18 $\pm$ 0.03&0.16   $\pm$ 0.08   \\
$1-\widehat{c}\ (\gamma=0.5)$& 0.10   $\pm$  0.01   &0.26 $\pm$ 0.04    &0.00 $\pm$ 0.00&0.22   $\pm$ 0.08   \\
$1-\widehat{c}\ (\gamma=1.0)$& 0.06   $\pm$  0.06   &0.32   $\pm$ 0.18    &0.00 $\pm$ 0.00&0.16   $\pm$ 0.20   \\
\end{tabular}
\vspace*{5mm}

\begin{tabular}{c|cccc}
                                    & \multicolumn{4}{c}{Outlier Probability for $(\x,y)$: $40\%$} \\ \cline{2-5}
  Methods                           & cal-housing        & abalone               &pumadyn-32fh   & bank-8fh  \\ \hline
   L2                                &312.92 $\pm$ 216.64&396.11 $\pm$ 135.79    & 4.36 $\pm$ 1.05&17.96 $\pm$ 5.83  \\
   L1                                &232.63 $\pm$ 264.15&\!\!276.02 $\pm$ 66.20 & 4.75 $\pm$ 1.04&14.06 $\pm$ 2.14  \\
   Huber                             &230.85 $\pm$ 257.14&\!\!278.29 $\pm$ 64.77 & 4.83 $\pm$ 1.13&13.84 $\pm$ 2.02  \\
   LTS                               &428.22 $\pm$ 432.53&\!\!111.59 $\pm$ 17.70 & 1.28 $\pm$ 0.33&0.48  $\pm$ 0.23  \\
   GemMc                             &0.47   $\pm$ 0.14  &2.73   $\pm$ 0.27      & 0.05 $\pm$ 0.03&0.08  $\pm$ 0.01  \\
   MM-est                            &0.50   $\pm$ 0.20  &2.50   $\pm$ 0.26      & 1.60 $\pm$ 0.54&0.08  $\pm$ 0.01  \\
$S_{\mathrm{power}}\ (\gamma=0.1)$   &0.47   $\pm$ 0.15  &2.49   $\pm$ 0.24      & 0.05 $\pm$ 0.03&0.08  $\pm$ 0.01  \\
$S_{\mathrm{power}}\ (\gamma=0.5)$   &0.48   $\pm$ 0.16  &2.56   $\pm$ 0.24      & 1.58 $\pm$ 0.86&0.08  $\pm$ 0.01  \\
$S_{\mathrm{power}}\ (\gamma=1.0)$   &0.51   $\pm$ 0.21  &2.57   $\pm$ 0.23      & 1.52 $\pm$ 0.76&0.08  $\pm$ 0.01  \\ \hline
$1-\widehat{c}\ (\gamma=0.1)$&0.38   $\pm$ 0.04  &0.39   $\pm$ 0.07      & 0.34 $\pm$ 0.02&0.18  $\pm$ 0.20  \\
$1-\widehat{c}\ (\gamma=0.5)$&0.32   $\pm$ 0.06  &0.43   $\pm$ 0.10      & 0.00 $\pm$ 0.00&0.17  $\pm$ 0.21  \\
$1-\widehat{c}\ (\gamma=1.0)$&0.12   $\pm$ 0.19  &0.26   $\pm$ 0.28      & 0.01 $\pm$ 0.09&0.02  $\pm$ 0.09  \\
\end{tabular}
 \caption{
 The numerical results on benchmark datasets are presented. 
 The training samples size is 100, and the contamination ratio $1-c_0$ is set to $0.05, 0.2$ or 
 $0.4$. Both the independent and dependent variables $(x,y)$ are contaminated by the outliers. 
 For each estimator, we employed the linear regression model with the intercept, and computed
 the averaged RMSE over 100 iterations. Also the estimated contamination ratio is presented. 
 }
 \label{tbl:benchmark-xy_out}
\end{table}

\section{Conclusion}
\label{sec:Conclusion}
In this paper, the robust statistical inference under heavy contamination is studied. 
In order to estimate not only the model parameter but also the contamination ratio, 
scoring rules such as the density-power score or pseudo-spherical score 
are applied with enlarged models. 
The proposed method is used for regression problems. 
Even under heterogeneous contamination, the proposed method with the location-scale
model provides an estimate of the expected contamination ratio besides a robust estimator
of the target model parameter. 
Using the estimator of the contamination ratio, 
one can identify the outliers out of the observed samples. 
Numerical experiments showed the effectiveness of our approach. 

As shown in~\cite{maronna06:_robus_statis,yu12:_polyn_form_robus_regres}, 
the convex loss function does not provide strong robustness to heavy contamination. 
This fact makes the optimization in the robust estimation harder. 
In the numerical experiments, the multi-start strategy is used 
as well as the other robust estimators. 
For the clipped loss function, Yu et al, \cite{yu12:_polyn_form_robus_regres} proposed the
relaxation approach for efficient computation. 
This approach is not directly available to our methods, 
since the loss functions proposed in this paper are not expressed as the form of the
clipped loss. 
A future work is to study numerical algorithms that are specialized for robust statistical
inference.

\appendix

\section{Preliminaries of Scoring Rules}
\label{appendix:preliminaries}
The density-power score and pseudo-spherical score are described as a special case of 
the H\"{o}lder score \eqref{eqn:Holder-score}. Indeed, 
the density-power score is derived from $\phi(z)=\gamma-(1+\gamma)z$, and the pseudo-spherical
score is derived from $\phi(z)=-z^{1+\gamma}$.  
First, we prove the inequality for H\"{o}lder score, $S_\phi(f,g)\geq{}S_\phi(f,f)$.
Then, we show the condition the equality for each score. 

Given non-negative functions $f$ and $g$, H\"{o}lder's inequality leads to 
\begin{align*}
 \<fg^\gamma\>\leq \<f^{1+\gamma}\>^{1/(1+\gamma)}\<g^{1+\gamma}\>^{\gamma/(1+\gamma)}
\end{align*}
for $\gamma>0$. The equality holds if and only if $f$ and $g$ are linearly dependent. 
From the inequality $\phi(z)\geq-z^{1+\gamma}$ for $z\geq0$, we have 
\begin{align*}
 S_\phi(f,g)-S_\phi(f,f)
 &=
 \phi\left(\frac{\<fg^\gamma\>}{\<g^{1+\gamma}\>}\right)\<g^{1+\gamma}\>+\<f^{1+\gamma}\>\\
 &\geq
 -\left(\frac{\<fg^\gamma\>}{\<g^{1+\gamma}\>}\right)^{1+\gamma} \<g^{1+\gamma}\>+\<f^{1+\gamma}\>\\
 &\geq0. \qquad\qquad\qquad\text{(H\"{o}lder's inequality)}
\end{align*}
Hence, the property of pseudo-spherical score was shown. 
For the density-power score, suppose that $S_{\mathrm{power}}(f,g)=S_{\mathrm{power}}(f,f)$ holds. 
Then, the inequalities in the above should become equality. 
The equality condition of H\"{o}lder's inequality leads that 
$f$ and $g$ are linearly dependent. For the function $\phi(z)$ of the density-power score, 
$\phi(z)=-z^{1+\gamma}$ holds only when $z=1$. Hence, 
$\<fg^\gamma\>/\<g^{1+\gamma}\>=1$ should hold. 
For non-negative and non-zero linearly dependent functions $f$ and $g$, 
the equality $\<fg^\gamma\>/\<g^{1+\gamma}\>=1$ leads to $f=g$.

\section{Proofs of Lemma~\ref{lemma:pseudo-spherical-Holder} and Theorem~\ref{theorem:Holder-min-alg}}
\label{appendix:proof_opt_sol}
\begin{proof}
 [Proof of Lemma~\ref{lemma:pseudo-spherical-Holder}]
 As defined in Section \ref{subsec:Holder}, 
 the function $\phi$ in the H\"{o}lder score satisfies $\phi(z)\geq-z^{1+\gamma}$. 
 Thus, we have $\psi_u(z)\geq-u^{1+\gamma}$ for $z>0$. 
 Since the inequality $\phi(z)>-z^{1+\gamma}$ is assumed for $z\neq-1$, 
 the equality $\psi_u(z)=-u^{1+\gamma}$ is satisfied only when $z=u$. 
 Hence, we have, 
 \begin{align*}
  S_\phi(\widetilde{p},cp_{\theta})
  =
%  c^{1+\gamma}
  \psi_{c(\theta)}(c)\<p_\theta^{1+\gamma}\>
  \geq
  -\<p_\theta^{1+\gamma}\>c(\theta)^{1+\gamma}, 
 \end{align*}
 and the equality holds only for $c=c(\theta)$. 
 In addition, $\psi_{c(\theta)}(c)$ is assumed to be strictly decreasing 
 on the open interval $c\in(0,c(\theta))$. 
 If $c(\theta)\leq1$ holds, clearly $c=c(\theta)$ is the optimal solution of 
 \eqref{eqn:unconstraint-Holder-opt}. Otherwise, 
 $c=1$ is optimal, since the inequality $1<c(\theta)$ assures that 
 $S_\phi(\widetilde{p},cp_\theta)$ is strictly decreasing with respect to $c$ over the
 interval $(0,1]$. 
 In summary, the optimal solution is expressed as $c=\min\{1,c(\theta)\}$. 
\end{proof}

\begin{proof}
 [Proof of Theorem~\ref{theorem:Holder-min-alg}]
 When $\widehat{c}=1$ holds, the statement of the theorem is clear. 
 Let us suppose that $0<\widehat{c}<1$ holds. 
 Note that the equality
 \begin{align}
  \label{eqn:Holder-pseudosphere-relation}
  S_\phi(\widetilde{p},c(\theta)p_\theta)=-(-S_{\mathrm{sphere}}(\widetilde{p},p_\theta))^{1+\gamma}
 \end{align}
 holds for $\theta\in\Theta$. 
 Due to Lemma~\ref{lemma:pseudo-spherical-Holder}, 
 $\widehat{c}=c(\widehat{\theta})$ should hold, 
 because under the assumptions of Lemma~\ref{lemma:pseudo-spherical-Holder}, 
 the optimal value of $c$ should be expressed as $\min\{1,c(\theta)\}$ for each
 $\theta\in\Theta$. 
 Let $\mathcal{N}\subset\Theta$ 
 be an open neighborhood of $\widehat{\theta}$ such that 
 $0<c(\theta)<1$ holds for all $\theta\in\mathcal{N}$. 
 Then, we have
 \begin{align*}
  \min_{c\in(0,1],\,\theta\in\Theta} S_\phi(\widetilde{p},cp_\theta)
  =  S_\phi(\widetilde{p},c(\widehat{\theta})p_{\widehat{\theta}})
  =  \min_{\theta\in\mathcal{N}}S_\phi(\widetilde{p},c(\theta)p_{\theta})
 \end{align*}
 Due to the equality \eqref{eqn:Holder-pseudosphere-relation}, 
 the minimization of $S(\widetilde{p},c(\theta)p_\theta)$
 on $\mathcal{N}$ is identical to 
 the minimization of the pseudo-spherical score $S_{\mathrm{sphere}}(\widetilde{p},p_\theta)$ 
 on  $\mathcal{N}$. 
 Therefore, $\widehat{\theta}$ is a local optimal solution of \eqref{eqn:min-pseudospherical}. 
 Generally, the set $\mathcal{N}$ cannot be replaced with $\Theta$, since
 the optimal solution of  $\min_{\theta\in\Theta} S(\widetilde{p},c(\theta)p_\theta)$ 
 may not satisfy the constraint $c(\theta)\leq1$. 
\end{proof}

\section{Proof of Theorem~\ref{theorem:small-bias}}
\label{appendix:small-bias}
\begin{proof}
 %[Proof of the first statement in Theorem~\ref{theorem:asymptotic_distribution}]\ 
 The point $\xi_0=(c_0,\theta_0)$ is the unique minimizer of $f_0(\xi)$. 
 In the same way as the proof of Theorem~\ref{theorem:Holder-min-alg}, 
 we can prove that the point $\xi_1=(c_1,\theta_1)$ is also the minimizer of
 $f_1(\xi)=S_{\mathrm{power}}(p,cp_\theta)$. 
 For $\xi=(c,\theta)$, the equality \eqref{eqn:diff-power-contamination} leads to 
 \begin{align*}
  f_1(\xi)\leq{}f_0(\xi)\leq{}f_1(\xi)+(1+\gamma)\varepsilon_\theta, 
 \end{align*}
 where the constraint $c_0, c\in(0,1]$ is used to derive the second inequality. 
 Then, $f_1(\xi_1)\leq{}f_0(\xi_0)$ should hold, since
 $f_1(\xi_1)\leq{}f_1(\xi_0)\leq{}f_0(\xi_0)$. 
% Otherwise, we have $f_1(\xi_0)\leq{}f_0(\xi_0)<f_1(\xi_1)$, that is the contradiction. 
 In addition, we have 
 \begin{align*}
%  f_1(\xi_1)\leq{}f_1(\xi_0)\leq{}
  f_0(\xi_0)\leq{}f_0(\xi_1)\leq{}f_1(\xi_1)+(1+\gamma)\varepsilon_1, 
 \end{align*}
 implying that $\xi_0, \xi_1\in\mathcal{N}$. Moreover, we obtain 
 \begin{align*}
  f_0(\xi_1)-(1+\gamma)\varepsilon_1\leq{}f_1(\xi_1)\leq{}f_0(\xi_0). 
 \end{align*}
 Taylor expansion of $f_0(\xi_1)$ around $\xi_0$ and the assumption on the Hessian matrix 
 yield that
 \begin{align*}
  f_0(\xi_0)+\frac{\delta}{2}\|\xi_0-\xi_1\|^2-(1+\gamma)\varepsilon_1\leq{}f_0(\xi_0). 
 \end{align*}
 Therefore, $\|\xi_0-\xi_1\|=O(\varepsilon_1^{1/2})$ holds. 
\end{proof}

\section{Proof of Theorem~\ref{theorem:asymptotic_distribution}}
\label{appendix:asymptotics}
\begin{proof}
[Proof of the first statement in Theorem~\ref{theorem:asymptotic_distribution}]\ 
 Suppose $0<c_0<1$. 
 Since $c(\widetilde{\theta})$ is assumed be a $\sqrt{n}$-consistent estimator of
 $c_1$,  %\memo{======reference???=======}
 the large deviation theory assures that the inequality
 $c(\widetilde{\theta})<1$ holds with the probability more than $1-e^{-\alpha{}n}$, 
 where $\alpha$ is a positive constant. 
 Therefore, the constraint $c\leq1$ in 
 \eqref{eqn:Holder-multiplicativemodel-opt} does not affect the asymptotic distribution 
 of the estimator $\widehat{\xi}=(\widehat{c},\widehat{\theta})$. 
 When $c(\widetilde{\theta})<1$, the estimator
 given by 
 $(\widehat{c},\widehat{\theta})=(c(\widetilde{\theta}),\widetilde{\theta})$
 does not depend on the choice of the function $\phi$, as shown in
 Theorem~\ref{theorem:Holder-min-alg}. 
 The density-power score is employed to calculate the asymptotic distribution of
 the estimator. 
 Note that the function $\phi$ of the density-power score satisfies 
 the assumptions in the theorem. 
 Suppose that the density-power score $S_{\mathrm{power}}(p,cp_\theta)$ is expressed as 
 \begin{align*}
  S_{\mathrm{power}}(p,cp_\theta)=\int{}p(x)\ell(x,cp_\theta)dx. 
 \end{align*}
 The minimum solution is $\xi_1=(c_1,\theta_1)$. 
 The asymptotic theorem of the M-estimator shows that the asymptotic distribution of 
 $\sqrt{n}(\widehat{\xi}-\xi_1)$ is the multivariate normal distribution with the mean 
 zero and variance-covariance matrix 
 $\Sigma_p=J_p^{-1}K_pJ_p^{-1}, $
 where the $d+1$ by $d+1$ matrices $J_p$ and $K_p$ are given as 
 \begin{align*}
  (J_p)_{ij}
  &\phantom{:}=
  \bigg\<p\cdot\frac{\partial^2}{\partial\xi_i\partial\xi_j}
  \ell(\cdot,cp_\theta)\bigg\>\bigg|_{\xi=\xi_1},\\
  \quad
  (K_p)_{ij}
  &\phantom{:}=
  \mathrm{Cov}_p\left[
  \frac{\partial}{\partial\xi_i}\ell(X,cp_\theta), 
  \frac{\partial}{\partial\xi_j}\ell(X,cp_\theta)
  \right]\bigg|_{\xi=\xi_1}\\
  &:=
  \bigg\<p\cdot
  \frac{\partial}{\partial\xi_i}\ell(\cdot,cp_\theta)
  \frac{\partial}{\partial\xi_j}\ell(\cdot,cp_\theta)
  \bigg\>
  -
  \bigg\<p\cdot  \frac{\partial}{\partial\xi_i}\ell(\cdot,cp_\theta)  \bigg\>
  \bigg\<p\cdot  \frac{\partial}{\partial\xi_j}\ell(\cdot,cp_\theta)  \bigg\>\bigg|_{\xi=\xi_1}.
\end{align*}
 Let $s_{\theta,i}$ be $\frac{\partial}{\partial\theta_i}\log{p_\theta(x)}$, and 
 $s_{\theta,ij}$ be
 $\frac{\partial^2}{\partial\theta_i\partial\theta_j}\log{p_\theta(x)}$, 
 and let $\Theta_0$ be a convex subset of $\Theta$ such that $\theta_1,\theta_0\in\Theta_0$.
Let us define 
$\bar{\varepsilon}_{\theta,ij}$ as
\begin{align*}
\bar{\varepsilon}_{\theta,ij}=
 \max\{
 %\<wp_{\theta}^{\gamma}\>,  
 \varepsilon_\theta,
\<wp_{\theta}^{2\gamma}\>,  
\<wp_{\theta}^{\gamma}|s_{\theta,i}|\>,  
\<wp_{\theta}^{2\gamma}|s_{\theta,i}|\>, 
\<wp_{\theta}^{\gamma}|s_{\theta,i}s_{\theta,j}|\>,  
\<wp_{\theta}^{2\gamma}|s_{\theta,i}s_{\theta,j}|\>, 
\<wp_{\theta}^{\gamma}|s_{\theta,ij}|\> 
 \}, 
\end{align*}
 and let 
 $\bar{\varepsilon}=\sup\{\bar{\varepsilon}_{\theta,ij}\,|\,i,j=1,\ldots,d,\,\theta\in\Theta_0\}$. 
 Using $\varepsilon_1=\varepsilon_{\theta_1}\leq\bar{\varepsilon}$, 
 we have $\|\xi_0-\xi_1\|=O(\bar{\varepsilon}^{1/2})$. Then, we obtain
\begin{align*}
 (J_p)_{ij}
 =
 (J_0)_{ij} +O(\bar{\varepsilon}^{1/2}),\qquad
 (K_p)_{ij}
 =
 (K_0)_{ij} +O(\bar{\varepsilon}^{1/2}), 
\end{align*}
 where the $(1+d)$ by $(1+d)$ matrices $J_0$ and $K_0$ are defined as 
\begin{align*}
 J_0=
 \bigg\<c_0p_{\theta_0}
 \cdot\frac{\partial^2}{\partial\xi_i\partial\xi_j}
 \ell(\cdot,cp_\theta)\bigg\>\bigg|_{\xi=\xi_0},
 \quad
 K_0=
 \mathrm{Cov}_{c_0p_{\theta_0}}\left[
  \frac{\partial}{\partial\xi_i}\ell(X,cp_\theta), 
  \frac{\partial}{\partial\xi_j}\ell(X,cp_\theta)
  \right]\bigg|_{\xi=\xi_0}. 
\end{align*} 
 In the above, we assumed that 
 the derivatives of $\<wp_\theta^\gamma\>$ and $\<wp_\theta^{2\gamma}\>$ 
 up to the third order on $\Theta_0$ are uniformly
 bounded by an integrable function. 
 As a result, the asymptotic variance-covariance matrix is given as 
 $\Sigma_p=\Sigma_{\xi_0}+O(\bar{\varepsilon}^{1/2})$, where
 $\Sigma_{\xi_0}$ is defined as $\Sigma_{\xi_0}=J_0^{-1}K_0J_0^{-1}$. 
\end{proof}

\begin{proof}
 [Proof of the second statement in Theorem~\ref{theorem:asymptotic_distribution}]\ 
 Suppose $c_0=1$. In this case, there is no outliers, and $p=p_0=p_{\theta_0}$ holds. 
 Hence, under the regularity conditions, 
 $(c(\widetilde{\theta}),\widetilde{\theta})$ converges to $\xi_0=(1,\theta_0)$
almost surely, and $\bar{\theta}$ also converges to $\theta_0$ almost surely. 
The asymptotic behaviour of $(c(\widetilde{\theta}),\widetilde{\theta},\bar{\theta})$ is 
obtained by using the asymptotic expansion. 
Let us define $\s_\theta$ as the $\Rbb^d$-valued score function
$\frac{\partial}{\partial\theta}\log{p_\theta(x)}$, and let $u$ and $\v$ be 
\begin{align*}
  u&=\frac{1}{\sqrt{n}}\sum_{i=1}^{n}\{p_{\theta_0}(x_i)^\gamma-\<p_{\theta_0}^{1+\gamma}\>\},\\
 \v&=\frac{1}{\sqrt{n}}
 \sum_{i=1}^{n}\{p_{\theta_0}(x_i)^\gamma\s_{\theta_0}(x_i)-\<p_{\theta_0}^{1+\gamma}\s_{\theta_0}\>\}. 
\end{align*}
 Then, the asymptotic expansion of the estimating equation for 
 $c(\widetilde{\theta}), \widetilde{\theta}$ and $\bar{\theta}$ yields that 
\begin{align*}
 \sqrt{n}
 \begin{pmatrix}  c(\widetilde{\theta})-1\\ \widetilde{\theta}-\theta_0 \end{pmatrix}
 &=\begin{pmatrix}
   \<p_{\theta_0}^{1+\gamma}\>     & \<p_{\theta_0}^{1+\gamma}\s_{\theta_0}^T\>\\ 
   \<p_{\theta_0}^{1+\gamma}\s_{\theta_0}\> & \<p_{\theta_0}^{1+\gamma}\s_{\theta_0}\s_{\theta_0}^T\>
   \end{pmatrix}^{-1}
 \begin{pmatrix}  u\\ \v \end{pmatrix}+o_p(1),\\
 \sqrt{n}(\bar{\theta}-\theta_0)
 &=
 \left(
 \<p_{\theta_0}^{1+\gamma}\s_{\theta_0}\s_{\theta_0}^T\>
 +\frac{\phi''(1)}{\gamma(1+\gamma)\<p_{\theta_0}^{1+\gamma}\>}
 \<p_{\theta_0}^{1+\gamma}\s_{\theta_0}\>\<p_{\theta_0}^{1+\gamma}\s_{\theta_0}\>
 \right)^{-1}\\
 &\phantom{=}\times
 \left( \frac{\phi''(1)}{\gamma(1+\gamma)}
 \frac{\<p_{\theta_0}^{1+\gamma}\s_{\theta_0}\>}{\<p_{\theta_0}^{1+\gamma}\>}u+\v \right)+o_p(1). 
\end{align*}
 The asymptotic expansion of $\sqrt{n}(\bar{\theta}-\theta_0)$
 is shown in the proof of Theorem~5 in~\cite{kanamoriar:_affin_invar_diver_compos_scores_applicy}. 
 The asymptotic probability densities of 
 $\sqrt{n}(c(\widetilde{\theta})-1,\,\widetilde{\theta}-\theta_0)$
 and 
 $\sqrt{n}(c(\widetilde{\theta})-1,\,\bar{\theta}-\theta_0)$
 % $(Z_0,\Z_1)$ and $(Z_0,\Z_2)$ 
 are respectively
 denoted as $p(z_0,\z_1)$ and $p(z_0,\z_2)$ for $z_0\in\Rbb,\,\z_1,\z_2\in\Rbb^d$, 
 in which the notation is overloaded. 
 Under the regularity condition, 
 $p(z_0,\z_1)$ and $p(z_0,\z_2)$ are $(1+d)$-dimensional normal distributions 
 with the mean zero. 
 Let us define $\widetilde{p}$ and $\bar{p}$ as 
 \begin{align*}
  \widetilde{p}(z_0,\z_1) =  2\,p(z_0,\z_1)\1[z_0\leq0],\qquad
  \bar{p}(\z_2) = 2\!\int{\!}\,p(z_0,\z_2)\1[z_0>0]dz_0, 
 \end{align*}
 where the indicator function $\1[A]$ takes $1$ if $A$ is true, and $0$ otherwise. 
 Informally, 
 $\widetilde{p}(z_0,\z_1)$ denotes the conditional probability density 
 $p(\z_0,\z_1|z_0\leq0)$, and 
 $\bar{p}(\z_2)$ denotes $p(\z_2|z_0>0)$. 
 % and $\bar{p}(\z_2)=p(\z_2|z_0>0)$
 The symmetry of the distribution assures that 
 the asymptotic probability such that $\sqrt{n}(c(\widetilde{\theta})-1)\leq0$  
 is equal to $1/2$. 
 Hence, the asymptotic probability density of 
 $\sqrt{n}(\widehat{\xi}-\xi_0)$
 % the estimator, $(z_0,\z)=\sqrt{n}(\widehat{c}-1,\widehat{\theta}-\theta_0)$, 
 is expressed as 
\begin{align*}
 \frac{1}{2}\widetilde{p}(z_0,\z)+\frac{1}{2}\delta(z_0) \bar{p}(\z). 
\end{align*}
 The first term is equal to $\phi_{1+d}(z_0,\z;\Sigma_{\xi_0})\1[z_0\leq0]$, that
 corresponds to the distribution of the estimator
 $(c(\widetilde{\theta}),\widetilde{\theta})$ in the case of 
 $c(\widetilde{\theta})\leq1$. 
 The second term corresponds to the distribution of the estimator $(1,\bar{\theta})$ in
 the case of $c(\widetilde{\theta})>1$.  The density $\bar{p}(\z)$ is expressed as the
 $d$-dimensional normal distribution with the mean zero and the variance-covariance matrix
 $\Lambda_{\theta_0}$ that is determined from the asymptotic expansions and the integral
 in the above. 
\end{proof}

\bibliographystyle{plain}
%\bibliography{allref}

\end{document}